\documentclass[11pt]{amsart}
\usepackage{mathrsfs,amsmath,amssymb,latexsym,dsfont,stmaryrd,enumerate}

\usepackage[all,2cell]{xy} \UseAllTwocells \SilentMatrices
\usepackage[T1]{fontenc}
\usepackage{xcolor}

\usepackage[pdftex]{hyperref}

\begin{document}

\newcommand{\INVISIBLE}[1]{}

\newtheorem{thm}{Theorem}[section]
\newtheorem{lem}[thm]{Lemma}
\newtheorem{cor}[thm]{Corollary}
\newtheorem{prp}[thm]{Proposition}
\newtheorem{conj}[thm]{Conjecture}

\theoremstyle{definition}
\newtheorem{dfn}[thm]{Definition}
\newtheorem{question}[thm]{Question}
\newtheorem{nota}[thm]{Notations}
\newtheorem{notation}[thm]{Notation}
\newtheorem*{claim*}{Claim}
\newtheorem{ex}[thm]{Example}
\newtheorem{counterex}[thm]{Counter-example}
\newtheorem{rmk}[thm]{Remark}
\newtheorem{rmks}[thm]{Remarks}
 
\def\labelenumi{(\arabic{enumi})}

\newcommand{\aro}{\longrightarrow}
\newcommand{\arou}[1]{\stackrel{#1}{\longrightarrow}}
\newcommand{\RA}{\Longrightarrow}

\newcommand{\mm}[1]{\mathrm{#1}}
\newcommand{\bm}[1]{\boldsymbol{#1}}
\newcommand{\bb}[1]{\mathbf{#1}}

\newcommand{\bA}{\boldsymbol A}
\newcommand{\bB}{\boldsymbol B}
\newcommand{\bC}{\boldsymbol C}
\newcommand{\bD}{\boldsymbol D}
\newcommand{\bE}{\boldsymbol E}
\newcommand{\bF}{\boldsymbol F}
\newcommand{\bG}{\boldsymbol G}
\newcommand{\bH}{\boldsymbol H}
\newcommand{\bI}{\boldsymbol I}
\newcommand{\bJ}{\boldsymbol J}
\newcommand{\bK}{\boldsymbol K}
\newcommand{\bL}{\boldsymbol L}
\newcommand{\bM}{\boldsymbol M}
\newcommand{\bN}{\boldsymbol N}
\newcommand{\bO}{\boldsymbol O}
\newcommand{\bP}{\boldsymbol P}
\newcommand{\bY}{\boldsymbol Y}
\newcommand{\bS}{\boldsymbol S}
\newcommand{\bX}{\boldsymbol X}
\newcommand{\bZ}{\boldsymbol Z}

\newcommand{\cc}[1]{\mathcal{#1}}

\newcommand{\ca}{\cc{A}}

\newcommand{\cb}{\cc{B}}

\newcommand{\cC}{\cc{C}}

\newcommand{\cd}{\cc{D}}

\newcommand{\ce}{\cc{E}}

\newcommand{\cf}{\cc{F}}

\newcommand{\cg}{\cc{G}}

\newcommand{\ch}{\cc{H}}

\newcommand{\ci}{\cc{I}}

\newcommand{\cj}{\cc{J}}

\newcommand{\ck}{\cc{K}}

\newcommand{\cl}{\cc{L}}

\newcommand{\cm}{\cc{M}}

\newcommand{\cn}{\cc{N}}

\newcommand{\co}{\cc{O}}

\newcommand{\cp}{\cc{P}}

\newcommand{\cq}{\cc{Q}}

\newcommand{\cR}{\cc{R}}

\newcommand{\cs}{\cc{S}}

\newcommand{\ct}{\cc{T}}

\newcommand{\cu}{\cc{U}}

\newcommand{\cv}{\cc{V}}

\newcommand{\cy}{\cc{Y}}

\newcommand{\cw}{\cc{W}}

\newcommand{\cz}{\cc{Z}}

\newcommand{\cx}{\cc{X}}

\newcommand{\g}[1]{\mathfrak{#1}}

\newcommand{\af}{\mathds{A}}
\newcommand{\PP}{\mathds{P}}

\newcommand{\GL}{\mathrm{GL}}
\newcommand{\PGL}{\mathrm{PGL}}
\newcommand{\SL}{\mathrm{SL}}
\newcommand{\NN}{\mathbf{N}}
\newcommand{\ZZ}{\mathbf{Z}}
\newcommand{\CC}{\mathbf{C}}
\newcommand{\QQ}{\mathbf{Q}}
\newcommand{\RR}{\mathds{R}}
\newcommand{\FF}{\mathbf{F}}
\newcommand{\DD}{\mathbf{D}}
\newcommand{\VV}{\mathds{V}}
\newcommand{\HH}{\mathds{H}}
\newcommand{\MM}{\mathds{M}}
\newcommand{\OO}{\mathds{O}}
\newcommand{\LL}{\mathds L}
\newcommand{\BB}{\mathds B}
\newcommand{\kk}{\mathbf k}
\newcommand{\bs}{\mathbf S}
\newcommand{\GG}{\mathds G}
\newcommand{\XX}{\mathds X}
\newcommand{\EE}{\mathds E}

\newcommand{\WW}{\mathds W}
\newcommand{\al}{\alpha}

\newcommand{\be}{\beta}

\newcommand{\ga}{\gamma}
\newcommand{\Ga}{\Gamma}

\newcommand{\om}{\omega}
\newcommand{\Om}{\Omega}

\newcommand{\vt}{\vartheta}
\newcommand{\te}{\theta}
\newcommand{\Te}{\Theta}

\newcommand{\ph}{\varphi}
\newcommand{\Ph}{\Phi}

\newcommand{\ps}{\psi}
\newcommand{\Ps}{\Psi}

\newcommand{\ep}{\varepsilon}

\newcommand{\vr}{\varrho}

\newcommand{\de}{\delta}
\newcommand{\De}{\Delta}

\newcommand{\la}{\lambda}
\newcommand{\La}{\Lambda}

\newcommand{\ka}{\kappa}

\newcommand{\si}{\sigma}
\newcommand{\Si}{\Sigma}

\newcommand{\ze}{\zeta}

\newcommand{\fr}[2]{\frac{#1}{#2}}
\newcommand{\vs}{\vspace{0.3cm}}
\newcommand{\na}{\nabla}
\newcommand{\pd}{\partial}
\newcommand{\po}{\cdot}
\newcommand{\met}[2]{\left\langle #1, #2 \right\rangle}
\newcommand{\hh}[3]{\mathrm{Hom}_{#1}(#2,#3)}
\newcommand{\spc}{\mathrm{Spec}\,}
\newcommand{\an}{\mathrm{an}}
\newcommand{\red}{\mathrm{red}}
\newcommand{\NNo}{\NN\smallsetminus\{0\}}
\newcommand{\Fr}{\mathrm{Fr}}


\newcommand{\Alg}[1]{#1\text{-}\mathbf{Alg}}
\newcommand{\modules}[1]{#1\text{-}\mathbf{mod}}
\newcommand{\modulesu}[1]{#1\text{-}\mathbf{modu}}
\newcommand{\comodulesu}[1]{ \mathbf{ucomod}\text{-}#1}
\newcommand{\comodules}[1]{ \mathbf{comod}\text{-}#1}
\newcommand{\pichopf}{\mathbf{PICHpf}}
\newcommand{\vect}[1]{#1\text{-}\mathbf{vect}}
\newcommand{\Modules}[1]{#1\text{-}\mathbf{Mod}}
\newcommand{\dmod}[1]{\mathcal{D}(#1)\text{-}{\bf mod}}
\newcommand{\rep}[2]{\mathbf{Rep}_{#1}(#2)}
\newcommand{\set}{\mathbf{Set}}

\newcommand{\pos}[2]{#1\llbracket#2\rrbracket}

\newcommand{\lau}[2]{#1(\!(#2)\!)}

\newcommand{\cpos}[2]{#1\langle#2\rangle}

\newcommand{\id}{\mathrm{id}}

\newcommand{\one}{\mathds 1}

\newcommand{\ti}{\times}
\newcommand{\tiu}[1]{\underset{#1}{\times}}

\newcommand{\ot}{\otimes}
\newcommand{\otu}[1]{\underset{#1}{\otimes}}

\newcommand{\wh}{\widehat}
\newcommand{\wt}{\widetilde}
\newcommand{\ov}[1]{\overline{#1}}
\newcommand{\un}[1]{\underline{#1}}

\newcommand{\op}{\oplus}

\newcommand{\lid}{\varinjlim}
\newcommand{\lip}{\varprojlim}

\newcommand{\mcrs}{\bb{MC}_{\rm rs}}
\newcommand{\mclog}{\bb{MC}_{\rm log}}
\newcommand{\mc}{\mathbf{MC}}

\newcommand{\ega}[3]{[EGA $\mathrm{#1}_{\mathrm{#2}}$, #3]}

\newcommand{\asts}{\begin{center}$***$\end{center}}
\title[Singular varieties and non-commutative  Witt vectors]{Singular varieties and  infinitesimal  non-commutative Witt vectors}

\author[Hai]{Ph\`ung H\^o Hai}
\address[Hai]{Institute of Mathematics, Vietnam Academy of Science and Technology, 18 Hoang Quoc Viet, Cau Giay, Hanoi}
\email{phung@math.ac.vn}

\author[dos Santos]{Jo\~ao Pedro dos Santos}
\address[dos Santos]{Institut Montpelli\'erain A. Grothendieck,  
Case courrier 051,
Place Eug\`ene Bataillon,
34090 Montpellier, France }
\email{joao\_pedro.dos\_santos@yahoo.com}
\author[Thinh]{\DJ\`ao V\u{a}n Th\d inh}
\address[Thinh]{Institute of Mathematics, Vietnam Academy of Science and Technology, 18 Hoang Quoc Viet, Cau Giay, Hanoi}
\email{daothinh1812@gmail.com}

\begin{abstract}Given a projective variety $X$  over an algebraically closed field $k$, M. V. Nori introduced in \cite{nori76} a group scheme $\pi(X)$ which accounts for principal bundles $P\to X$ with finite structure,  obtaining in this way an amplification the etale fundamental group.  One   drawback  of this theory is that it is quite difficult to arrive at an explicit  description of $\pi(X)$, whenever it does not vanish altogether. To wit, there are no known non-trivial examples in the literature where  $\pi(X)$ is  local,  or local of some given height, etc. In this paper we obtain a    description of $\pi(X)$ through amalgamated products of certain non-commutative local group schemes  --- we called them infinitesimal non-commutative Witt group schemes --- in the case where $X$ is a non-normal variety  obtained by  pinching a simply connected one.
\end{abstract}


\thanks{
The research of Ph\`ung H\^o Hai and \DJ\`ao V\u an Th\d inh is funded by the Vietnam Academy of Science and Technology under grant number CBCLCA.01/25-27. Parts of this project were carried out during P. H. Hai's several visits  to the International Centre for Theoretical Physics (ICTP):  he   thanks this institution  for its hospitality and support. 
}
\thanks{The research of \DJ\`ao V\u an Th\d inh was supported by the Postdoctoral program of Vietnam Institute for Advanced Study in Mathematics and by a    ``Poste Rouge'' of the CNRS   at the University of Montpellier.}
\date{20 June 2025}

\maketitle

\section{Introduction}\label{introduction}

In this work we are concerned with  finite group schemes which appear as structure groups of principal bundles over projective varieties in positive characteristic, that is to say, with the theory of the fundamental group scheme initiated by M. V. Nori \cite{nori82,nori76}. To begin our discussion, given a projective variety $X$ over an algebraically closed field $k$ of characteristic $p>0$, let us write $\pi(X,x_0)$ for the essentially finite fundamental group scheme of $X$ at     $x_0\in X(k)$ \cite[Section 3]{nori76}. 
One of the drawbacks of this object  
--- and for that matter also of its predecessor, the  etale fundamental group ---   is that a   {\it precise}  calculation turns out to be quite difficult. 
Indeed, although there are many instances where one can assure that  $\pi(X,x_0)$ is  {\it trivial} \cite{biswas09,biswas-holla07,mehta11,biswas-kumar-parameswaran24}, or at least {\it small} \cite{antei-biswas16,biswas-hai-dos_santos21}, it is far from being an easy task to find actual descriptions outside of these cases. For example, the present work shall give the {\it first}  known example  for which $\pi(X,x_0)$ is {\it certainly} local  and non-trivial, for which $\pi(X,x_0)$ is of height precisely $1$, etc.  
(If we abandon the requirement of projectivity  imposed on  $X$, it is worth noticing that describing $\pi(X,x_0)$ precisely seems somewhat pointless, and putting the question as an existence theorem  ``\`a la Abhyankar''  is  more reasonable \cite{otabe18,otabe-tonini-zhang22}.) 

One of the difficulties behind the determination of $\pi(X,x_0)$ is:   what we ultimately  want is to identify when $\pi(X,x_0)$  {\it already appears  somewhere else}. Now, prime examples of pro-finite group schemes are (those associated to) pro-finite groups  and  Frobenius kernels of linear algebraic groups. Many among these examples are applications of the following construction principle, which we now explain.   We start with a finite abstract group or  a restricted Lie algebra, we construct the group ring,      respectively the universal enveloping algebra, note that these are in fact cocommutative Hopf algebras 
 (see \cite[Example 2.7, p.62]{abe80} and 
 \cite[\S6]{milnor-moore65}), and apply the   {\it Hopf dual} functor $(-)^\circ$  \cite{abe80,sweedler69} to arrive at the ring of functions of an  group scheme. 
 Thus, we may enquire when $\co(\pi(X,x_0))\simeq H^\circ$, where $H$ is a ``known'' cocommutative Hopf algebra. 
This point of view has the advantage of being more economic and  simpler to apply to Tannakian categories;  it is easier  to deal with modules over an algebra (viz. $H$) than with  comodules over a coalgebra (viz. $\co(\pi(X,x_0))$). It also has the advantage of opening a door to  Dieudonn\'e's theory of formal groups \cite{dieudonne73}. 

As one learns from \cite{serre88}, \cite{brion15} and \cite[Expos\'e IX, \S5]{SGA1}, 
introducing non-normal points on a regular variety has the effect of producing certain {\it accessory groups} on the starting one.   More precisely, we assume   the existence of a smooth variety $Y$ and a finite morphism \[\nu:Y\aro X\]   
which allow us   to express sheaves on $X$ in terms of sheaves on $Y$ plus some extra data; here we have in mind  the process of pinching \cite{serre88,ferrand02}. Then, 
given a certain kind of group associated to a geometric object (see ahead), 
it is possible to describe the groups associated to $X$ in terms of those of $Y$ {\it plus}  certain {\it  accessory groups}. In the context of Jacobian varieties, the accessory groups are either powers of $\mathbf G_a$ or products of Witt groups \cite[V.13--16]{serre88}. In the theory of Picard schemes  \cite{brion15},  the accessory groups   are Weil restrictions, and  in \cite[IX.5]{SGA1} we encounter pro-finite  amalgamated products $\wh{\mathbf Z}*\cdots*\wh{\mathbf Z}$. It then becomes clear that for the case of  the Nori fundamental group scheme an {\it obviously  challenging part is  to find  analogues of these accessory groups}.  Indeed, in our setting, well-known results allow  us to rapidly transpose the geometry to a purely algebraic problem controlled by certain associative and cocommutative Hopf algebras $\bm H_D$  (see Section \ref{14.02.2025--1}). These Hopf algebras (or their Hopf dual)  play the role of the accessory groups, and the task is then to see if these objects already have a meaning somewhere else.

  This is one point which the present work  addresses by bringing to light the  ``non-commutative  infinitesimal Witt group schemes'' discovered independently by E. J. Ditters and K. Newman more than 50 years ago \cite{ditters69,newman74}. (Ditters draws on the outstanding series of papers Dieudonn\'e produced on formal groups in the 1950's and 60's, while Newman is influenced by Sweedler.) Unfortunately, although very interesting, these group schemes have received little attention in overview works. The   cocommutative  Hopf algebras discovered by these two mathematicians ---   {\it  non-commutative exponential group-coalgebras} in \cite{ditters69},   $\g P_n$ in \cite{newman74}, but $\cc{NW}_{n-1}$ in Definition \ref{nonCommuattiveWitt} --- are not  found in the textbook presentations of Hopf algebras      at our disposal and are only tangentially mentioned, in the context of formal groups, at the  end of  \cite[VI.38]{hazewinkel78}. 
  
  Once these ``non-commutative Witt group schemes'' are under the light, we see that 
they have a role similar to the one played by   Witt groups in the theory of the Jacobian   \cite[V.16]{serre88} and are the ``accessory groups'' we were looking for. Indeed, this is argued, in a  different setting, by  Ditters and  Newman \cite{ditters69,newman74}. In a nutshell, the ``excess'' appearing in the Jacobian mentioned by Serre (cf. V.13  and Proposition 9 of V.16 in \cite{serre88}) has a manifestation in the theory of the fundamental group scheme.

Since our objective is   to  describe as precisely as possible the group scheme  $\pi(X,x_0)$, we shall restrict attention to the case where {\it essentially finite vector bundles on $Y$} are rather simple,   i.e. they are always {\it trivial}. By doing so, the bulk of the work consists in understanding the accessory group schemes mentioned above. It is not excluded that a more encompassing picture can be reached by allowing a more general $Y$, but this seems more delicate (see e.g. Example \ref{05.03.2025--1})  and is to  be dealt with somewhere else.

Let us now review the remaining sections of the paper. 
Section \ref{subsidiary} has a three-fold function. Firstly, it surveys briefly the paper \cite{deninger-wibmer23} (see Section  \ref{categorical_fibre_products}) so that we have the theory of fibre-products of tensor categories and 
amalgamated products of group schemes right from the start. Secondly, it gathers, in  Section \ref{public_utility1},  material explaining how the ring of regular functions $\co(G')$ of a certain group scheme $G'$ {\it associated} to an affine  group scheme $G=\spc H^\circ$ is more accessible when written as a Hopf dual.  This part is   employed in Section \ref{26.02.2025--2}, but has some independent interest as well. Finally, in Section \ref{public_utility2}, we explain  the construction of the coproduct  in the category of Hopf algebras;
this construction is a key ingredient in \cite{newman74},  which is fundamental to Section \ref{07.05.2025--1}, but we were unable to find convenient      references in the literature. 

In Section \ref{19.02.2025--1}, we review briefly the process of pinching a closed finite subscheme   $D\to Y$ to a smaller finite scheme $C$ to obtain a variety $X=Y\sqcup_DC$. The main reference for the theme here is \cite{ferrand02}. This is used mostly to gather notations and terminology. 
 In Section \ref{vector_bundles_on_pinchde_schemes}
 we bring to light Milnor's theorem and its consequences in the description of vector bundles on the schemes previously obtained. We then introduce the categories $\ct_{D,C}$, one of the  main actors of the paper because of their accessibility, and explain the first decomposition results for  these (see Proposition \ref{14.02.2025--2}). 
 
 Section \ref{17.04.2025--1} continues with the breaking up of the categories $\ct_{D,C}$ into smaller pieces, but here the hypothesis are strengthened and  $C$ is taken to be {\it reduced}. In this section, the category $\cs_D$, which will produce  via Tannakian duality  the  
 first ``accessory'' group scheme $\Sigma_D$,  appears. (See Definition \ref{24.02.2025--1}.) 
 Now, if  in Section \ref{17.04.2025--1} the main categorical decompositions are obtained, it is in Section \ref{23.04.2025--1} that the group theoretical results  {\it and} the explicit calculations of $\pi(X)$  appear. In it, by introducing the hypothesis that $\pi(Y)=0$, we are capable of providing a decomposition of $\pi(X)$ 
in terms of the group schemes  $\Si_L$, where $L$ is a  connected component of $D$, see Theorem \ref{22.04.2025--1}.   
Although more can be said about these $\Sigma_L$, this section bears already some fruits: we obtain clear examples of schemes $X$ such that $\pi(X)$ is solely local and of a given finite height (see Proposition \ref{25.02.2025--2}). 
Let us mention that $\Sigma_L$ is the analogue of Serre's ``$V_{(n)}$'',  which plays a role in describing the Jacobian of a singular curve  \cite[V.13--16]{serre88}.

Section \ref{05.02.2025--1} reviews the works of Ditters and Newman and adapts it to our setting. It also provides the group schemes which can be read off from \cite{ditters69,newman74}
a more expressive  name, the non-commutative (infinitesimal) Witt group schemes $\boldsymbol{N\!W}_\ell$, see Definition \ref{nonCommuattiveWitt}. The justification for such a term is offered by Proposition  \ref{23.04.2025--3}.  
 
 Section \ref{14.02.2025--1} makes use of an important theorem of Newman in order to decompose the group scheme   $\Sigma_D$ of Definition \ref{24.02.2025--1}, when $D$ is connected,  into
 infinitesimal non-commutative Witt group schemes, see Corollary  \ref{23.04.2025--3}. It is a non-commutative  analogue of the well-known decomposition of the groups $V_{(n)}$ of \cite[V.13--16]{serre88} into Witt groups, cf. \cite[V.16]{serre88}.
 This then finishes our description of $\pi(X)$ (see Corollary \ref{23.04.2025--5}).

The reader will notice that the text contains many reviews and preliminary material. This is mainly due to the varied nature of the techniques ---  tensor categories, pinching of schemes, vector bundles and Hopf algebras ---  employed to obtain a definite result. Requiring the reader to feel 
confortable with all these branches in order to understand the calculation  of a fundamental group scheme seemed a poor decision. Hopefully this shall prove useful in appreciating  the points of view of other mathematicians.

\subsection*{Some notations and conventions}

\begin{enumerate}[(1)]

\item $k$ stands for an  algebraically closed  field of characteristic $p>0$, {\it except} in Section \ref{subsidiary}, where no assumptions are necessary. The category of finite dimensional vector spaces is denoted by 
 $\vect k$. The scheme  $\spc k$ is denoted by $\mathbf{pt}$. 
\item All rings are associative and unital. When discussing   algebraic geometric aspects, we shall also assume that the rings in question are commutative. 
\item Given a scheme $X$, a locally free $\co_X$-module of finite  rank is called a vector bundle. The category of vector bundles is denoted by $\bb{VB}(X)$.   
\item The Frobenius morphism of a scheme or of a ring is denoted by 
$\mm{Fr}$. 
\item Given an ideal $\g a$ in a commutative ring $A$ of characteristic $p>0$, we let $\g a^{[p^h]}$ denote the ideal generated by $\mm{Fr}^h(\g a)$. 
\item Given a vector space $V$ or a set   set $S$,  we shall denote the free $k$-algebra on $V$, respectively 
$S$, by $\bm T(V)$ or $k\{V\}$, respectively $k\{S\}$.    
\item Given an associative algebra $R$, respectively a coalgebra $L$, over $k$, we let $\modules R$, respectively $\comodules L$, stand for the category of left $R$-modules, respectively right $L$-comodules, which are {\it finite dimensional over $k$}.

\item The category of {\it Hopf agebras} over $k$ is denoted by $\bb{Hpf}$. That of cocommutative Hopf algebras is denoted by  $\mathbf{Hpf}^{\mathrm{coc}}$.  (This differs from the   tradition in the theory of formal groups, where this same category is denoted by   $\bb{GCog}$   \cite{dieudonne73}, or $\bb{GCoalg}$  \cite{ditters69}.)

\item  We let  $\bb{GS}$ stand for the category of affine group schemes over $k$. For brevity, objects in $\bb{GS}$ are referred to simply as {\it group schemes}. 

\item For a group scheme $G$, we shall make no distinction between representations of $G$ and right comodules over  $\co(G)$ \cite[Chapter 3]{waterhouse79}.

\item Given an abelian group $\La$, we let $\mm{Diag}(\La)$ stand for the diagonal group scheme constructed as in \cite[Part I, 2.5]{jantzen03}: its algebra of functions is the group algebra $k\La$. 

\item We follow the terminology of \cite[\S1]{deligne-milne82} concerning {\it tensor categories}. 

\item Given an integral and proper $k$-scheme $X$, we shall denote by $\bb{EF}(X)$ the category of essentially finite vector bundles on $X$.  Given a $k$-point $x_0$ of $X$, we let $\pi(X,x_0)$ stand for the group scheme obtained from $\bb{EF}(X)$ via the fibre functor $\bullet|_{x_0}:\bb{EF}(X)\to\vect k$. 
See \cite{nori76}. 
\end{enumerate}

\section{Subsidiary material}\label{subsidiary}
We remind the reader that in this section no assumption on the field  $k$ is to be imposed. 
\subsection{Fibre products and amalgamated products}\label{categorical_fibre_products}
Many of the arguments in this work  rely on the notion of fibre product of categories, which we recall briefly,   for the convenience of the reader, and   to establish notations.

Let $F:\ca\to\cs$, $G:\cb\to\cs$ be functors between categories and let \[\cp=\ca\ti_\cs\cb\] be the category whose objects are triples \[(a,b; \ga)\quad\text{ with $a\in\mm{Obj}\,\ca$, $b\in\mm{Obj}\,\cb$ and $\ga:Fa\stackrel\sim\to Gb$,}\]  and whose arrows are   defined in the obvious manner, cf.   \cite[1.6]{ferrand02} or \cite[1.1]{deninger-wibmer23}. Since the case where $\cs$ is the category $\vect k$ will be mostly used, we will adopt the abbreviation 
\[
\boxed{\ca\ti_k\cb:=\ca\ti_{\vect k}\cb.}
\]

When dealing with       tensor categories \cite[\S1]{deligne-milne82} $\ca$, $\cb$ and  $\cs$, {\it and} tensor functors $F$ and $G$, the category  $\cp$ can be endowed with a functor $\cp\ti\cp\to\cp$,    
 \[
 (A, B; c) \ot (A', B'; c') = (A\ot  A', B \ot B'; c \ot c'),
 \]
which then turns $\cp$ into a tensor category. 
See \cite[Lemma 1.5]{deninger-wibmer23}. In addition, rigidity is preserved in this construction,  as is ``abelianess'' \cite[Lemma 1.4]{deninger-wibmer23}, provided that $F$ and $G$ are exact. 

Taking  one   step further, let us suppose that $\ca$, $\cb$ and $\cs$ are {\it neutralised Tannakian categories}. This means that each one of them is an abelian, $k$-linear  tensor category \cite[1.15]{deligne-milne82}, which is rigid \cite[1.7]{deligne-milne82}, and that   there exist  tensor functors $\xi:\ca\to \vect k$, $\ze:\cb\to \vect k$ and $\eta:\cs\to\vect k$ which are faithful, $k$-linear and exact.  Suppose now that 
the functors $F$ and $G$ preserve these structures: there exist isomorphisms of tensor functors $\xi \stackrel\sim\Rightarrow\eta F$ and $\ze \stackrel\sim\Rightarrow\eta G$. Then $\cp$ is a $k$-linear abelian tensor category and comes with an exact and faithful tensor functor $\xi\ti\ze:\cp\to\vect k$ \cite[Lemma 1.8]{deninger-wibmer23}. All pieces are in place in order to apply the main existence theorem \cite[Theorem 2.11]{deligne-milne82}.

\begin{dfn}
Let $G$ and $H$ be   group schemes and  endow $\mathbf{Rep}_k(G)$ and $\mathbf{Rep}_k(H)$ with the forgetful functors $\om_G:\mathbf{Rep}_k(G)\to\vect k$ and $\om_H:\mathbf{Rep}_k(H)\to\vect k$. 
  We define \[G\star H\] as the   group scheme, associated via \cite[Theorem 2.11]{deligne-milne82}, to the couple consisting of the rigid abelian tensor category  \[\mathbf{Rep}_k(G)\ti_k\mathbf{Rep}_k(H)\] and the functor  \[\om_G\ti \om_H:\mathbf{Rep}_k(G)\ti_k\mathbf{Rep}_k(H)\aro \vect k.\]
\end{dfn}

From \cite[Corollary 1.12]{deninger-wibmer23}, it follows that $G\star H$ is the {\it co-product, in the category of group schemes $\bb{GS}$}, of $G$ and $H$. (The co-product will also be called the amalgamated product.)

\subsection{On the Hopf dual}\label{public_utility1}  
Let $R\in\mathbf{Hpf}^{\mathrm{coc}}$; the   Hopf algebra  $R^\circ$  \cite{abe80,sweedler69} is commutative and   $G:=\spc R^\circ$   then comes with the structure of  a group scheme. Let now $\bb A$ be a full subcategory of     $\bb{GS}$   for which it is possible to construct the {\it largest quotient  in $\bb A$} \cite[Section 2.2]{hai-dos_santos18}; let us agree to denote this quotient by   $G\to G^{\bb A}$. 
We may then enquire if $G^{\bb A}$ {\it can be described in terms of $R$ exclusively}. This 
is an useful  entreprise for the present work and, in addition,  has a certain public utility, so that   we have enough reason to include the following lines dealing with the case where $\bb A$ is either the category of   {\it abelian} or  {\it unipotent}  group schemes. It should be noticed that the specific case  of group algebras $R=k\Ga$,  where $\Ga$ is an abstract group and $\mm{char}\,k=0$, is much used in Quillen's theory of the Malcev completion, cf. \cite[Section 3]{hain93}. 

Let   $\bb{Ab}$, respectively $\bb {Un}$, be the categories of {\it abelian}, respectively {\it unipotent} \cite[Chapter 8]{waterhouse79},    group schemes. According to the general theories of    group-schemes and categories  (see \cite[Chapter 8]{waterhouse79} and  \cite[Section 2.2]{hai-dos_santos18}), the construction of the largest quotient is possible. 

Let $\ep:R\to k$ be the counit of $R$ and $\g a$ its kernel, the augmentation ideal.  Let $\modules R$ stand for the category of all left $R$-modules whose dimension as a  $k$-space  is finite; in $\modules R$ we have one preferred object, $\mathbf I:=R/\g a$, called the unit object. This terminology is   coherent because the fact that $R$ is a cocommutative Hopf algebra allows us to endow $\modules R$ with a tensor structure and, in this case, $\bb I$ is a unit object.  This is well explained in \cite[Proposition 1.1]{schneider95}.  
Similarly, let $\comodules {R^\circ}$ stand for the category of right $R^\circ$-comodules having finite dimension over $k$. Using   \cite[2.1.3]{montgomery93} and  \cite[1.6.4]{sweedler69},  the identity functor on $\vect k$ gives rise to an  isomorphism  of $k$-linear tensor categories 
\[\Te:
\modules R\arou\sim\comodules{R^\circ}=\mathbf{Rep}_k\,G. 
\]
In particular $\bb I$  corresponds to the trivial representation of  $G$.

Let \[R^\dagger=\{f\in R^\circ\,:\,\text{$f(\g a^n)=0$ for some $n\ge1$}\}\] and 
  $R^{\mathrm{ab}}=R/\g C$, where $\g C$ is the two-sided ideal generated by the commutators in $R$.

\begin{lem}\label{public_utility2}\begin{enumerate}[(1)]
\item The subspace $R^\dagger\subset R^\circ$ is a Hopf subalgebra. 
\item 
 The $k$-algebra $R^{\mm{ab}}$ has a unique structure of Hopf algebra such that $R\to R^{\mm{ab}}$ is a morphism of Hopf algebras.

\item  The natural morphism  $G=\spc R^\circ\to\spc R^{\mm{ab},\circ}$ is the largest abelian quotient of $G$.
\item The natural morphism  $G=\spc R^\circ\to\spc R'$ is the largest unipotent quotient of $G$.
\end{enumerate}
\end{lem}
\begin{proof}(1) This is   \cite[Lemma 9.2.1]{montgomery93}. 

(2) We only need to observe that $\g C$ is a Hopf ideal and apply  standard theory \cite[4.2.1]{abe80}. 

(4)
Let    $\modulesu R$ be the  full subcategory of    $\modules R$ consisting of those modules for which  Jordan-H\"older factors are all   isomorphic to $\mathbf I$. It is not difficult, using simple linear algebra, to see that under the isomorphism $\Te$ mentioned above,  the category  $\modulesu R$ corresponds to  the category of {\it unipotent} comodules $\comodulesu{R^\circ}$, that is, those $R^\circ$-comodules whose Jordan-H\"older factors are all isomorphic to $\one$. We contend that $M\in\modulesu R$ if and only if some power of $\g a$ annihilates $M$.  
Suppose then that $M\in\modulesu R$ and let $0=M_0\subset M_1\subset\ldots\subset M_n=M$ be a Jordan-H\"older filtration where  $M_i/M_{i-1}\simeq\bb I$. Then, for each $a\in\g a$, it follows that  $a\cdot M_i\subset M_{i-1}$,  and hence   $a_1\cdots a_n  M=0$ whenever  $a_1,\ldots,a_n\in \g a$, i.e. $\g a^nM=0$. 
Conversely, let   $M\in\modules R$ be   annihilated by    $\g a^n$. Then   $\g a^n$ annihilates any simple module $S$ appearing in a Jordan-H\"older filtration of $M$. Now,  $\g p:=\{r\in R\,:\,rS=0\}$  is a  two-sided  {\it prime ideal} \cite[3.12, p.56]{goodearl-warfield04}  containing $\g a^n$ and hence $\g p\supset \g a$. This implies that $\g p=\g a$ and $S\simeq \bb I$.

In conclusion, only the unit representation of  $G^\dagger:=\spc R^\dagger$ is  simple   and is hence $G^\dagger$ is unipotent. In addition, if $f:G\to U$ is a morphism to an unipotent group scheme $U$, then $\mathbf{Rep}_{k}\,U\to\mathbf{Rep}_{k}\,G\simeq \modules R$ must have its image in $\modulesu R$ and therefore,  
 Tannakian duality   leads to a factorization of $f$ into $G\to G^\dagger \to U$.  

(3) The proof is simple and omitted. 
\end{proof}

\subsection{Coproduct of Hopf algebras}\label{public_utility2} 
We review here the construction of the coproduct in the category of {\it Hopf algebras}; we begin with a revision of the coproduct of $k$-algebras. 

Let $A$ and $B$ be $k$-algebras (associative and  unital). From \cite[Section 1.4]{BMM96}, we know that the coproduct of $A$ and $B$ 
exists {\it in the category of $k$-algebras}. More precisely, there exists a $k$-algebra $P$ and morphisms $i:A\to P$, and $j:B\to P$ such that a diagram of full arrows as below can be completed to a commutative diagram by a {\it unique} dotted arrow:  
$$\xymatrix{
& R&\\ 
A\ar[r]_-{i}\ar[ru]^{\varphi}& P \ar@{-->}[u]& B. \ar[l]^-{j}
\ar[lu]_{\ps}
}$$
The construction of $P$ in \cite[Section 1.4]{BMM96} shows in particular that  $P$ is generated by $i(A)\cup j(B)$. By definition, $P$ is determined uniquely up to unique isomorphism.

We now suppose that $A$ and $B$ are {\it bialgebras}; let 
  $\De_A$ and $\ep_A$ be the comultiplication and counit, and adopt similar notations for $B$.  Then, the morphisms of $k$-algebras $(i\ot i)\circ\De_A$
and $(j\ot j)\circ\De_B$ give rise to a morphism of $k$-algebras $\De_P:P\to P\ot P$ such that \[\De_P\circ i=(i\ot i)\circ\De_A\quad\text{and}\quad\De_P\circ j=(j\ot j)\circ\De_B.\]
Analogously, we have   morphisms of $k$-algebras $\ep_P:P\to k$ such that  $\ep_P\circ i=\ep_A$  and $\ep_P\circ j=\ep_B$. 
Since $P$ is generated by $i(A)\cup j(B)$, a direct calculation shows that 
\[
(\De_P\ot\id_P)\circ\De_P=(\id_P\ot\De_P)\circ\De_P\quad\text{and}\quad(\ep_P\ot\id_P)\circ\De_P=\id_P=( \id_P\ot \ep_P )\circ\De_P.
\]
This gives $P$ the structure of a coalgebra, and since $\De_P$ and $\ep_P$ are morphisms of algebras, $P$ is then a {\it bialgebra}. In addition,  $i:A\to P$ and $j:B\to P$ are morphisms of bialgebras. 

The same method assures   that we can give a   step further. Assume that $A$ and $B$ are {\it Hopf algebras}, with antipodes  $S_A:A\to A$ and $S_B:B\to B$; they are in particular anti-morphisms of $k$-algebras. Then there exists an anti-morphism of $k$-algebras  $S_P:P\to P$ such that $S_P\circ i=S_A$ and $S_P\circ j=S_B$. Now, the two equations necessary to show that $S_P$ is an antipode for $P$ \cite[Exercises 1, p. 73]{sweedler69} are verified for elements in $i(A)\cup j(B)$, and it is a simple exercice to show that if these equations hold for $k$-algebra generators of $P$, then they hold for all elements. We conclude that $P$ comes with a natural structure of Hopf algebra. 

To end, we observe that the Hopf algebra $P$ is a coproduct {\it in the category of Hopf algebras} for $A$ and $B$; this can  again be checked using that  $i(A)\cup j(B)$
generates the $k$-algebra $P$.

\section{Pinching a modulus   }\label{19.02.2025--1}  In this section, we explain the construction of the main space studied in the paper. Notations shall also be employed in other sections. 

 Let $Y$ be an integral and  {\it projective} scheme over $k$. Inspired by  Serre \cite[III.1]{serre88}, we define   a {\it modulus} on $Y$  as   any closed immersion $\te:D\to Y$ of a {\it finite} $k$-scheme $D$. 
Let $C$ be another finite $k$-scheme and let $\mu:D\to C$ be a schematically dominant morphism, i.e. $\mu^\#:\co(C)\to\co(D)$ is injective. From \cite[Theorem 5.4]{ferrand02}, there exists a  scheme  $X$ (see Remark \ref{24.04.2025--1} below) and a {\it co-cartesian diagram}
\begin{equation}\label{the_diagram}
\xymatrix{
D \ar [d]_\mu \ar@{^{(}->}[rr]^\te   && Y   \ar [d]^\nu 
\\
C\ar@{^{(}->}[rr]_\si   && X;
}
\end{equation}
the scheme $X$ shall be referred to as {\it the pinching of $D$ to $C$ in $Y$}.  
 In this case, $\si$ is a closed immersion, $\nu$ is a finite morphism and 
\[\nu:Y\setminus D\aro X\setminus C\]
is an isomorphism (in particular $\nu$ is also surjective). As Ferrand explains on \cite[6.1]{ferrand02}, the scheme $X$ is proper over $k$ as well. The case when $C$ is a point appears in \cite[IV.4]{serre88}; it was recently employed in  \cite{das24}.

As we assume $Y$ to be projective,  the same can be said about $X$. (Here we rely on the finiteness of $D$, cf. \cite[Section 6]{ferrand02}.) Indeed, let $\cl$ be an ample invertible sheaf on $Y$. Since $D$ is finite, the locally free $\co_D$-module $\te^*\cl$ is {\it free} of rank one. Consequently, using the equivalence $\bb{VB}(X)\simeq \bb{VB}(C)\ti_{\bb{VB}(D)}\bb{VB}(Y)$ (this shall be explained   below), it follows that $\cl$ descends to an invertible sheaf on $X$, which must then be ample by the criterion of \cite[tag 0B5V]{stacks_project}. 

It is important to note that by allowing $D$, $C$ and $\mu$ to vary substantially, this construction can be used to obtain   projective varieties from their normalisations. Loosely speaking,  the above process can be reversed. Let $V$ be any projective variety over $k$ and  let $S$ be the closed subset of  non-normal points of $V$. Let $n:W\to V$ be its normalisation; $W$ is a normal and projective variety, $n$ is a finite morphism and $n:W\setminus n^{-1}(S)\to V\setminus S$ is an isomorphism.  Note that the coherent $\co_V$-module $n_*\co_W/\co_V$ is supported at $S$.
Let $\cC=\mm{Ann}_{\co_V}(n_*\co_W/\co_V)$ be the conductor; it is  an ideal of $\co_V$ whose image $n^\#\cC\subset n_*(\co_W)$ is also an ideal  and hence $\cC$ gives rise to a closed subscheme $F\subset W$ \ega{II}{}{1.4}, a closed subscheme $E\subset V$ and a schematically dominant morphism $F\to E$. More importantly, the following diagram of rings  on $V$ 
\[
\xymatrix{\co_V\ar@{->>}[d]\ar@{^{(}->}[r]^-{n^\#}&n_*(\co_W)\ar@{->>}[d]
\\
\co_V/\cC\ar@{^{(}->}[r]&n_*(\co_W)/n^\#\cC}
\]
is {\it cartesian} by the ``conductor square'' (see pp. 556-7 in \cite{ferrand02}). From this and  Scholium 4.3 of \cite{ferrand02}, it follows that \[\xymatrix{ \ar[d]F\ar@{^{(}->}[r]&W\ar[d]^n\\ E\ar@{^{(}->}[r]&V}\] is cocartesian in the category of ringed spaces and a fortiori in the category of schemes. Finally, it should be noticed that   $E\subset V$ is supported on $S$ but may well fail to be reduced. 

\begin{ex} Let        $X\subset\spc k[x,y]=\bb A^2$ be   the singular curve    cut out by   $y^2-x^5=0$.    Let $\mathbf A^1=\spc k[t]$ and define 
$\nu:\mathbf A^1 \to X$ by  $a\mapsto(a^2,a^5)$. 
This gives rise to  an isomorphism $\mathbf A^1\setminus\{0\}\stackrel\sim\to X\setminus\{0\}$.
Hence,  
\[
\co(X)\simeq    k[t^2,t^5] \simeq k[t]\underset { k[t]/(t^4)}{\times} k[t^2]/(t^4).
\]
So, $X$ is obtained from $\mathbf A^1$  by pinching the modulus $\spc k[t]/(t^4)$ to $\spc k[t^2]/(t^4)$. Clearly   the  conductor   is $(t^4)$.
\end{ex}

\begin{rmk}\label{24.04.2025--1}There is point in the phrasing of   \cite[Theorem 5.4]{ferrand02} which can lead to confusion: Ferrand first   constructs the coproduct    $Y\sqcup_DC$ \cite[Scholium 4.3]{ferrand02} in the category $\bb{RS}$ of ringed spaces and then shows that $Y\sqcup_DC$ is a scheme. Now, the category of schemes is {\it not} a full subcategory of $\bb{RS}$! Hence, the possibility that a morphism  $Y\sqcup_DC\to T$ in $\bb{RS}$ obtained from morphisms of {\it schemes} $Y\to T$ and $C\to T$ fails to be a morphism of   {\it locally } ringed spaces must be avoided. Luckily this is the case. 
\end{rmk}

\begin{rmk}[Rosenlicht-Serre curves]Let us add to the assumptions made on   $Y$ that it is regular and of dimension one, i.e. $Y$ is a  regular and proper  curve. Then, under the assumption that $C$ is reduced, the curve $X=Y\sqcup_DC$ is of a certain special kind. S. Das, in \cite[Definition 2.1]{das24}, calls these   ``Rosenlicht-Serre curves'' because of its use in \cite{serre88}.
\end{rmk}

\section{Vector bundles on pinched schemes}\label{vector_bundles_on_pinchde_schemes}
We shall keep the notations and setting introduced in section \ref{19.02.2025--1} and summarised in the following commutative diagram 
\[
\xymatrix{
D \ar [d]_\mu \ar@{^{(}->}[rr]^\te   && Y   \ar [d]^\nu 
\\
C\ar@{^{(}->}[rr]_\si   && X.
}
\]

\subsection{The category $\ct_{D,C}$ and Milnor's Theorem}
Theorem 2.2(iv) in  \cite{ferrand02}, due to Milnor,   allows us to describe    $\bb{VB}(X)$ from  $\bb{VB}(Y)$ and $\bb{VB}(D)$. To be more precise, we shall employ the categorical fibre product $\bb{VB}(C)\ti_{\bb{VB}(D)}\bb{VB}(Y)$ constructed via  the functors $\te^*:\bb{VB}(Y)\to\bb{VB}(D)$ and $\mu^*:\bb{VB}(C)\to\bb{VB}(D)$ (cf. Section \ref{categorical_fibre_products}).   
Milnor's Theorem can then be applied to show that the natural functor 
\[
\bb{VB}(X)\aro \bb{VB}(C)\ti_{\bb{VB}(D)}\bb{VB}(Y)
\]
described in terms of objects by 
\[\ce\longmapsto(\si^*\ce,\nu^*\ce ;   \mu^* \si^* \ce\arou{\text{can}}\te^*\nu^*\ce ),
\]
is an equivalence. 
(In \cite[Theorem 2.2(iv)]{ferrand02}, the author presents a proof in the case of affine schemes;  the transposition to the  general   case is simple  as $D$ lies in an affine subscheme of $Y$. See \cite[Theorem 3.13]{howe} for details.)

Since our objective is  to  describe $\pi(X,x_0)$ as precisely as possible --- and as explained in Section \ref{introduction}, little is known about $\pi(Y)$ in general ---    we shall concentrate not on $\bb{VB}(X)$ itself, but on the full subcategory  
\begin{equation}\label{22.04.2025--3} \bb{VB}(C) \tiu{\bb{VB}(D)}\bb{VB}^{\mathrm{tr}}(Y),\end{equation}
  of $\bb{VB}(C)\ti_{\bb{VB}(D)}\bb{VB}(Y)$ consisting of those objects   whose ``second  component'' is trivial. This is done with the intention of imposing $\pi(Y)=0$ further ahead (see Section \ref{24.04.2025--2}). 
  To ease notation,  we  make the following definition. 

\begin{dfn}Let   $\ct_{D,C}$ be the  category $\bb{VB}(C)\ti_{\bb{VB}(D)}\vect k$. That is, $\ct_{D,C}$ has as  
\begin{enumerate}\item[\it{objects}] those triples $(V,W;\be)$, where $V\in\bb{VB}(C)$, $W\in\vect k$, and $\be:W\ot_k\co_D\to \mu^*V$ is an isomorphism. And an 
\item[\it{arrow}] from $(V,W;\be)$ to $(V',W';\be')$ is a couple $(f,g)\in\mm{Hom}_{\co_C}(V,V')\ti\mm{Hom}(W,W')$   rendering commutative the diagram 
\[
\xymatrix{\mu^*V   \ar[rr]^{\mu^*f}&&\mu^*V'   
\\
W\ot_k \co_D  \ar[u]^{\be}\ar[rr]^{g\ot\id}&&\ar[u]_{\be'}W'\ot_k\co_D. }
\]
\end{enumerate}
\end{dfn}

The category $\ct_{D,C}$ comes with a $k$-bilinear functor \[\ot:\ct_{D,C}\times\ct_{D,C}\aro \ct_{D,C}\]
defined by
\[
(V,W;\be)\ot(V',W';\be')=(V\ot_{\co_C} V',W\ot_k W';\tau(\be,\be')), 
\]
where $\tau(\be,\be')$ renders commutative the obvious diagram. This turns $\ct_{C,D}$ into a tensor category. 
We let   
\begin{equation}\label{19.02.2025--2}
\om_{D,C}:\ct_{D,C}\aro\vect k, 
\end{equation}
be defined  on the level of objects by $(V,W;\be)\mapsto W$. Note that we have an equality $\ot_k\circ(\om_{D,C}\ti\om_{D,C})=\om_{D,C}\circ\ot$. 
 In what follows, the functor $\om_{D,C}$ shall be referred to as  {\it the  fibre functor} of $\ct_{D,C}$. 

Finally, we observe that 
\begin{equation}\label{22.04.2025--2}
\ct_{D,C}\aro \bb{VB}(C)\tiu{\bb{VB}(D)} \bb{VB}^{\mathrm{tr}}(Y),\qquad(V,W;\be)\longmapsto(V,W\ot_k\co_Y;\be^{-1})
\end{equation}
gives rise to an equivalence     since $H^0(Y,-):\bb {VB}^{\mathrm{tr}}(Y)\to\vect k$ is an equivalence.
In what follows, $\ct_{D,C}$ is the object of our studies.
Our method consists in breaking up $\ct_{D,C}$ into fibered products over $\vect k$  and we begin by the simplest one in Section \ref{14.02.2025--3} below. 

\subsection{The case where $C$ is a disjoint union}\label{14.02.2025--3}We now suppose that $C$ can be written as a disjoint union of two closed and open subschemes, that is, for $i=1,2$, we have closed immersions $\ph_i:C_i\to C$ such that $C_1\sqcup C_2\simeq C$. 
We adopt the notations implied in the following {\it cartesian diagram:}
\[\xymatrix{
D_i\ar[rr]^-{\ps_i}\ar[d]_{\mu_i}&&D\ar[d]^\mu
\\
C_i\ar[rr]_{\ph_i}&&C.
}\] 
We then define functors 
\[
F_i:\ct_{D,C}\aro\ct_{D_i,C_i}
\]
as follows. Let $(V,W;\be)\in\mathrm{Obj}\,\ct_{D,C}$ and write   
 $F_i(V,W;\be)=(\ph_i^*V,W;\be_i)$, where  $\be_i$ is obtained from the commutative diagram  
\[
\xymatrix{
\ps_i^*(W\ot_k\co_D)\ar[rr]^-{\ps_i^*(\be)}&&\ps_i^*\mu^*V\ar[r]^{\text{can.}}&\mu_i^*(\ph_i^*V)
\\
\ar[u]^{\text{can.}}\ar[rrru]_{\be_i}W\ot_k\co_{D_i}.&&&
}
\]
Note that the composition 
\[
\xymatrix{\ct_{D,C}\ar[r]^{F_i}&\ct_{D_i,C_i}\ar[rr]^-{\om_{D_i,C_i}}&&\vect k}
\]
is just $\om_{D,C}$. Using the notion of 2-commutative diagram (see \cite[Section 1]{deninger-wibmer23}) we then arrive at a functor   \[G:\ct_{D,C}\aro \ct_{D_1,C_1}\ti_k\ct_{D_2,C_2},\] 
through \cite[Lemma 1.3]{deninger-wibmer23}. 
Note that $G(V,W;\be)=((\ph_1^*V,W,\be_1),(\ph_2^*V,W,\be_2);\id_W)$.

\begin{prp}\label{14.02.2025--2}
The functor $G$ is an equivalence of categories. 
\end{prp}
\begin{proof}We make first a general remark.   Let $\bm V=(V,W;\be)\in\ct_{D,C}$ and   $a:W' \to W$ be an isomorphism. Define $\be':W'\ot_k\co_D\to \mu^*V$ as the composition
\[\xymatrix{
\co_D\ot W'\ar[rr]^{\id\ot a}&&\co_D\ot_kW\ar[r]^-{\be}&\mu^*V.
}\]
We have therefore a new object $\bm V'=(V,W';\be')$, and   $(\id_V,a):\bm V'\to \bm V$ is isomorphism. Consequently, in the isomorphism class of an object  \[\bm P=((V_1,W_1;\be_1),(V_2,W_2;\be_2)\,;\,a:W_1\stackrel\sim\to W_2)\] 
from $\ct_{D_1,C_1}\ti_k\ct_{D_2,C_2}$, we always can find and object of the   form \[\bm P'=((V_1,W_1;\be_1),(V_2,W_1;\be_2')\,;\,\id_{W_1}).\]
This observation, together with the fact that $(\ph_1^*,\ph_2^*):\bb {VB}(C)\to \bb{VB}(C_1)\ti\bb{VB}(C_2)$ and $(\ps_1^*,\ps_2^*):\bb {VB}(D)\to \bb{VB}(D_1)\ti\bb{VB}(D_2)$ are equivalences, proves that $G$ is essentially surjective. 
\end{proof}

\section{Description of the category $\ct_{D,C}$ when $C$ is reduced}\label{17.04.2025--1}

We shall keep the assumptions and notations of Section \ref{vector_bundles_on_pinchde_schemes}. 
We   assume further  that  $C$ is {\it reduced} and set out to describe in more    $\ct_{D,C}$. This restriction on $C$ is an important one, but   produces a   clear picture: we shall soon see that the basic building blocks of our analysis of $\ct_{D,C}$ are certain
categories $\cs_\bullet$ associated to connected components of $D$ (see Definition \ref{24.04.2025--3}). These simpler categories will then produce by Tannakian duality the group schemes  $\Sigma_\bullet$ (see Definition   \ref{24.02.2025--1}), which    appeared, in a different form, in        \cite{ditters69,ditters75,newman74}. The case when  $C$ is not reduced should be dealt with in   future work.

Let  $C=C_1\sqcup\cdots\sqcup C_m$, where   $C_i=\mathbf{pt}=\spc k$. As in Section \ref{14.02.2025--3}, let $D_i$ stand for the inverse image of $C_i$ inside $D$.  
From Proposition \ref{14.02.2025--2}, we conclude that 
\[
\ct_{D,C}\simeq \ct_{D_1,\mathbf{pt}}\ti_k\cdots\ti_k\ct_{D_m,\mathbf{pt}}.
\]
It is therefore necessary to pay attention to the case where  $C=\mathbf{pt}$. 

\subsection{Description of $\ct_{D,\mathbf{pt}}$ for   $D$  connected }

We assume in this section that $
D$ is connected, i.e. $\co(D)$ is a local ring, and let $s:\bb{pt}\to D$ be its unique point. 
Given an isomorphism $f:W\to W'$ of vector spaces, we have an isomorphism \begin{equation}\label{21.02.2025--1}
(\id_V,f):(V,W;\be)\arou\sim(V,W';\be\circ(f^{-1}\ot_k \id_{\co_D})  )
\end{equation}
and, in particular, 
each    $(V,W;\be)\in\mm{Obj}\,\ct_{D,\mathbf{pt}}$ is isomorphic to some $(V,V,\be')$ by letting $f=s^*(\be)$. This shows how to break up $\ct_{D,\mathbf{pt}}$ into smaller pieces with the help of:

\begin{dfn}\label{24.04.2025--3}   Define  $\cs_D^+$   as the category whose  
\begin{enumerate}
\item[\it objects]  are $(V,\be)$, with $V \in\vect k$ and $\be:V\ot \co_D \to V\ot\co_D$ an isomorphism,   and  
\item[\it arrows] from $(V ,\be)$ to $(V' ,\be')$ are maps $f\in  \mm{Hom}_k(V,V')$   rendering  
\[
\xymatrix{V\ot \co_D\ar[d]_\be\ar[rr]^{f\ot\id}&&V'\ot \co_D\ar[d]^{\be'}
\\
V\ot \co_D \ar[rr]^{f\ot \id}&&V'\ot \co_D }
\]
commutative. 
\end{enumerate}
Analogously, we define $\cs_D$ as being the full subcategory of $\cs_D^+$ whose objects are the couples $(V,\be)$ where, in addition, $s^*(\be) =\id_V$. 
\end{dfn}

Not surprisingly, 
$\cs_D^+$ and $\cs_D$ have structures of tensor  categories, where 
\[(V,\be)\ot(V',\be')=(V\ot_k V,\tau(\be,\be'))\]
and $\tau(\be,\be')$ renders commutative the diagram 
\[\xymatrix{
(V\ot_kV')\otu k\co_D\ar[d]_{\text{can.}} \ar[rr]^-{\tau(\be,\be')}&&(V\ot_kV')\otu k\co_D\ar[d]^{\text{can.}}
\\
(V\ot_k\co_D)\otu{\co_D}(V'\ot_k\co_D)\ar[rr]_{\be\ot\be'}&&(V\ot_k\co_D)\otu{\co_D}(V'\ot_k\co_D).
}\]
We remark that the forgetful  functor 
\begin{equation}\label{25.02.2025--3}
\sigma_D:\cs_D^+\aro\vect k
\end{equation} is a faithful and $k$-linear tensor functor. 
We shall soon see that   $\cs_D^+$ plays a supporting role with respect to $\cs_D$ (Proposition \ref{21.02.2025--4}).

There is an obvious functor 
\begin{equation}\label{18.02.2025--2}F:\cs_D \aro\ct_{D,\mathbf{pt}}\end{equation}
defined on the level of objects by    $(V,\be)\mapsto(V,V;\be)$ and on the level of arrows  by 
\[\mm{Hom}((V,\be),(V',\be'))\aro\mm{Hom}(F(V,\be),F(V',\be')),\qquad f\longmapsto(f,  f ).\]
From the isomorphism \eqref{21.02.2025--1} above, it follows that  $F$ is essentially surjective. A simple verification shows that $F$ is also fully faithful and hence the  proposition below holds true: 

\begin{prp}The  functor $F:\cs_D\to\ct_{D,\mathbf{pt}}$ of eq. \eqref{18.02.2025--2} is an equivalence of  $k$-linear categories. \qed
\end{prp}

We now  study  the relation between $\cs_D^+$ and  $\cs_D$.
Let    \[\ci:=\rep k\ZZ.\]   
For each $\bm V:=[(V,\be),(W,g);\ph]\in\cs_D\ti_k\ci$, we let 
\[
G\bm V =(V,\be  (\ph^{-1}   g\ph)_D )\in\mm{Obj}\,\cs_D^+; 
\]
here and below, we   abbreviate    $\xi\ot \id_{\co_D}$ to $\xi_D$.
Also, define for each arrow $(f,h)$ in $\cs_D\ti_k\ci$, the arrow $G(f,h)=f$. This gives rise to a functor $G:\cs_D\ti_k\ci\to\cs_D^+$.

\begin{prp}\label{21.02.2025--4}The  above defined    functor $G:\cs_D\ti_k\ci\to\cs_D^+$ is an equivalence of tensor categories.
\end{prp}
\begin{proof} This is a sequence of simple verifications. 
Let $(V,\ga)\in\mm{Obj}\,\cs_D^+$. Then, $(V,\ga  \circ(s^*(\ga)^{-1})_D )$ is an object of $\cs_D$ and $(V,s^*\ga)$ is an object of $\ci$. If we put  
\[\bm V:=
(V,\ga  (s^*(\ga)^{-1})_D\, ,\,(V,s^*\ga)\,;\,\id_V]\in\mathrm{Obj}\,\cs_D\ti_k\ci,\] 
then $G\bm V=(V,\ga)$,   and $G$ is essentially surjective. 

We now  check that $G$ is fully faithful.  Each $[(V,\be),(W,g);\ph]$ is isomorphic, via $(\id,\ph^{-1})$, to $[(V,\be),(V,\ph^{-1}g\ph);\id_V]$ and we may then work solely with  objects of this simplified kind.  
Let $\bm V=[(V,\be),(V,g);\id]$ and $\bm V'=[(V',\be'),(V',g');\id]$ be given  in $\cs_D\ti_k\ci$,  and let 
$\ph:V\to V'$ satisfy    
\begin{equation}\label{24.02.2025--3} \ph_D  \be  g_D=\be'  g_D' \ph_D;\end{equation}
i.e.  $\ph\in\mm{Hom}_{\cs_D^+}(G\bm V,G\bm V')$. Applying $s^*$ to eq. \eqref{24.02.2025--3}, we have  
\begin{equation}\label{24.02.2025--4}\ph g= g'\ph,\end{equation} i.e. $\ph\in\mm{Hom}_\ci((V,g),(V',g'))$. From eqs. \eqref{24.02.2025--3} and \eqref{24.02.2025--4}, we conclude  $\ph_D\be g_D=\be'\ph_Dg_D$ and thus   $\be'\ph_D=\ph_D\be $. This says that   $\ph\in\mm{Hom}_{\cs_D}( (V,\be), (V',\be'))$ and $(\ph,\ph)$ is an arrow $\bm V\to\bm V'$   inducing $\ph\in\hh{\cs_D^+}{G\bm V}{G\bm V'}$. 
Consequently,  $G_{\bm V,\bm V'}:\mm{Hom}(\bm V,\bm V')\to\mm{Hom}(G\bm V,G\bm V')$ is surjective. Injectivity is even simpler and we do not write it down. Nor shall we write down the verifications concerning the final statement, and we conclude here the proof. 
\end{proof}

\subsection{Determination of $\ct_{D,\mathbf{pt}}$ for arbitrary $D$}\label{21.02.2025--2}
 We now assume that $D=D_1\sqcup \cdots\sqcup D_m$, where each $D_i$ is connected and $m\ge2$. Let $\tilde D:=D_2\sqcup\cdots\sqcup D_m$ so that $D=D_1\sqcup \tilde D$. We identify $\bb{VB}(D)$ with $\bb{VB}(D_1)\ti\bb{VB}(\tilde D)$; in particular, maps between $\co_D$-modules shall be expressed as {\it pairs  of maps}. We then pick a point    $s:\mathbf{pt}\to \tilde D$  of $\tilde D$.
 Finally, let $d:D\to\bb{pt}$, $d_1:D_1\to\bb{pt}$ and $\tilde d:\tilde D\to\bb{pt}$ be the structural morphisms.
 
 The reader is asked to bear in mind that
$\ct_{\tilde D,\mathbf{pt}}$ and $\cs_{D_1}^+$  
  come equipped with functors to $\vect k$ (cf. eq. \eqref{19.02.2025--2} and eq. \eqref{25.02.2025--3}), which allow us to define $\cs^+_{D_1}\ti_k\ct_{\tilde D,\mathbf{pt}}$. 
    We set out to show that $\ct_{D,\bb{pt}}\simeq\cs^+_{D_1}\ti_k\ct_{\tilde D,\mathbf{pt}}$.

Given  an object  $[(U,\al),(V,W;\tilde \be);\ph]$ of $\cs^+_{D_1}\ti_k\ct_{\tilde D,\mathbf{pt}}$, we wish to associate to it an object of $\ct_{D,\mathbf{pt}}$.
Note that we are already in possession of an isomorphism $\tilde\be:\tilde d^*W\stackrel\sim\to \tilde d^*V$ and we need an isomorphism of $\co_{D_1}$-modules $\be_1:d_1^*W\stackrel\sim\to d_1^*V$. 
 Now, \[\xymatrix{U\ar[r]^{\ph}&W\ar[r]^{s^*(\tilde\be)}&V}\] is an isomorphism, and hence it is only natural that we define $\be_1$ by rendering 
\[\xymatrix{
d_1^*(U)\ar[d]_\al\ar[rr]^{d_1^*(\ph)}&&d_1^*(W)\ar[d]^{\be_1}
\\
d_1^*(U)\ar[rr]_{d_1^*(s^*(\tilde\be)\ph)}&&d_1^*(V)
}\]
commutative, i.e. \[\be_1=d_1^*(s^*(\tilde\be)\ph)\circ\al \circ d_1^*(\ph^{-1}).\] Therefore, for each $[(U,\al),(V,W;\tilde \be);\ph]$ we introduce the object \[H[(U,\al),(V,W;\tilde \be);\ph]=(V,W\,;\,d_1^*(s^*(\tilde\be)\ph)\circ\al\circ d_1^*(\ph^{-1}),\tilde\be)\] of $\ct_{D,\mathbf{pt}}$.   Also, given   any   arrow in $\cs_{D_1}^+\ti_k\ct_{\tilde D,\mathbf{pt}}$, 
\[
(f,g,h):[(U,\al),(V,W;\tilde\be);\ph]\aro[(U',\al'),(V',W';\tilde\be');\ph'], 
\]
an intricate  but straightforward computation shows that 
\[(g,h):(V,W;\be_1,\tilde\be) \aro (V',W';\be'_1,\tilde\be') \]
is an arrow in $\ct_{D,\mathbf{pt}}$ and in this way we arrive at a {\it faithful} and $k$-linear  functor \[H:\cs_{D_1}^+\ti_k\ct_{\tilde D,\mathbf{pt}}\aro\ct_{D,\mathbf{pt}}.\]

\begin{prp}\label{21.02.2025--3} The functor $H$ defined above is an equivalence of neutralized tensor categories. 
\end{prp}
\begin{proof}This is again a series of straightforward long verifications. {\it Essential surjectivity.} Let $\bm V=(V,W;\be_1,\tilde\be)\in\mm{Obj}\,\ct_{D,\mathbf{pt}}$ where, as before, $\be_1:d_1^*W\to d_1^*V$ and   $\tilde\be:\tilde d^*W\to\tilde d^*V$ are isomorphisms. Let $\bm V_1=(W,d_1^*s^*(\tilde\be^{-1}) \circ \be_1)\in\mm{Obj}\,\cs^+_{D_1}$ and  $\tilde{\bm V}=(V,W;\tilde\be)\in \mm{Obj}\,\ct_{\tilde D,\mathbf{pt}}$. Then \[H(\bm V_1,\tilde{\bm V};\id_W)=\bm V.\]

{\it Fully faithfulness.} Faithfulness has already been justified above. Each $\bm V\in\mm{Obj}\,\cs_{D_1}^+\ti_k\ct_{\tilde D,\mathbf{pt}}$ is isomorphic to an object of the form $[(W,\al),(V,W;\tilde\be);\id_W]$ so that in what follows we can restrict attention to these.  Let $\bm V=[(W,\al),(V,W;\tilde\be);\id_W]$ and $\bm V'=[(W',\al'),(V',W';\tilde\be');\id_{W'}]$ be objects of $\cs_{D_1}^+\ti_k\ct_{\tilde D,\mathbf{pt}}$ and let \[(g,h):H\bm V\aro H\bm V'\]
be an arrow of $\ct_{D,\mathbf{pt}}$. Hence,  \begin{equation}\label{03.03.2025--1}d_1^*(g)\circ d_1^*(s^*(\tilde\be))\circ\al=d_1^*(s^*(\tilde\be'))\circ\al'\circ d_1^*(h)\end{equation}  and \begin{equation}\label{03.03.2025--2}\tilde d^*(g)  \circ \tilde\be =\tilde\be'\circ\tilde d^*(h).\end{equation}From eq. \eqref{03.03.2025--2} we conclude that $g\circ s^*(\tilde\be)=s^*(\tilde\be')\circ h$, which together with eq. \eqref{03.03.2025--1} shows that $d_1^*(h)\circ\al=\al'\circ d_1^*(h)$. But this means that     $h:W\to W'$ defines an arrow of $\cs_{D_1}^+$ from $(W,\al)$ to $(W',\al')$ and we conclude that $(g,h)$ comes from an arrow $\bm V\to\bm V'$.

\end{proof}

Using Proposition \ref{21.02.2025--3} and then Proposition \ref{21.02.2025--4}, we have:

\begin{cor}\label{22.04.2025--4}If we decompose $D$ into connected components, $D=D_1\sqcup\cdots\sqcup D_m$,  then  the tensor category $\ct_{D,\mathbf{pt}}$ is equivalent to $\cs_{D_1}$ if $m=1$ and to 
\[
\left(\ci\ti_k\cs_{D_1}\right)\ti_k\cdots\ti_k\left(\ci\ti_k\cs_{D_{m-1}}\right)\ti_k\cs_{D_m} 
\] if $m\ge2$. 
In addition, the equivalences preserve  the canonical  functors to $\vect k$. \qed
\end{cor}

\section{ Description of $\pi(X)$ when $C$ is   reduced}
\label{23.04.2025--1}

 We  write the findings of  Section \ref{17.04.2025--1} in terms of group schemes in order to express $\pi(X)$
 as an amalgamated product. One of the pieces in this   product, the group scheme \[\Si_D\] obtained from $\cs_D$ when $D$ is connected,  turns out to be a fundamental block and will then be analysed in Section \ref{14.02.2025--1}.  
 
\subsection{Translation in terms of  Tannakian group schemes}
 We assume that $D$ is connected, except for Theorem \ref{22.04.2025--1} below.

\begin{dfn}\label{24.02.2025--1}We let $\Sigma_D$ be  the group scheme obtained from $\cs_D$ through the fibre functor $\si_D:\cs_D\to\vect k$ and Tannakian duality \cite[Theorem 2.11]{deligne-milne82}. Similarly, $\Si^+_D$ is the   group scheme obtained from $\cs_D^+$. 
\end{dfn}

 Let $\bb Z^{\rm alg}$ be the  group scheme associated to $\ci=\rep k\ZZ$ via the forgetful functor. Letting $\ZZ_p$ be the pro-finite group scheme $\lip_\ell\ZZ/p^\ell$, it is not difficult to show that $\bb Z^{\rm alg}=\mm{Diag}(k^*)\times\ZZ_p$. 
 More precisely, denote by $[m]:k^*\to k^*$ the map $a\mapsto a^m$. Then, under the identification $\mm{Diag}(k^*)(k)=\hh{\bb{Grp}}{k^*}{k^*}$ \cite[Part I, 2.5]{jantzen03}, it follows that the morphism of groups \[\ZZ\aro\mm{Diag}(k^*)(k)\times\ZZ_p,\qquad m\longmapsto([m],m)\]
induces an equivalence between representation categories. (It should be noted that $\ZZ^{\mathrm{alg}}$ is not necessarily pro-finite.)
From \cite[Corollary 1.11]{deninger-wibmer23} and Proposition \ref{21.02.2025--4} we conclude that 
\begin{equation}\label{25.02.2025--4}\Sigma_D^+
 \simeq \Sigma_D\star\left(\mm{Diag}(k^*)\times\ZZ_p\right).\end{equation}
 
Let us note the following result, whose proof shall be given in Section \ref{21.02.2025--7}. (It rests on the fact that $\co(\Si_D)$ can be very easily described as a Hopf dual.) 
The reader is asked to bear in mind our blanket assumption that $D$ is connected.

\begin{prp}\label{25.02.2025--2}The group scheme $\Sigma_D$ is local.   In fact, let $\g m$ stand for the maximal ideal of $\co(D)$ and let $h\ge1$ be such that $\g m^{[p^h]}=0$ while $\g m^{[p^{h-1}]}\not=0$. Then $\Sigma_D$ is annihilated by $\mm{Fr}^h$ but not by $\mm{Fr}^{h-1}$, i.e. is of height $h$ \cite[II.7.1.4--6]{demazure-gabriel70}.  \qed
\end{prp}  

Because each  local group scheme is pro-fintie (cf. Remark \ref{25.02.2025--5})  we have. 

\begin{cor}\label{25.02.2025--1}The group scheme $\Sigma_D$ is pro-finite. \qed
\end{cor}

With a view towards the computation of $\pi(X)$ in Section \ref{24.04.2025--2} below, let us now determine the {\it largest} pro-finite quotient $(\Si_D^+)^{\mm{pf}}$ of $\Si_D^+$. (For this notion, see Section \ref{public_utility1}.) 
Being   a {\it left-adjoint}, the functor $(-)^{\mm{pf}}$   must commute with colimits \cite[V.5]{maclane98} and hence eq. \eqref{25.02.2025--4} jointly with Corollary \ref{25.02.2025--1} give us
\[
\begin{split}
\left(\Si_D^+ \right)^{{\rm pf}}&\simeq \Sigma_D^{\rm pf}\star\left(\mm{Diag}(k^*)\times\ZZ_p\right)^{{\rm pf}}\\&\simeq\Sigma_D\star\left(\mm{Diag}(k^*)\times\ZZ_p\right)^{{\rm pf}}.
\end{split}\] 
 Now,  $(k^*)_{\rm tors}\simeq \varinjlim_{m}\ZZ/m$,  the limit being taken over all positive integers   which are not divisible by $p$,  so that        
\[\left(\mm{Diag}(k^*)\times\ZZ_p\right)^{{\rm pf}}\simeq \wh\ZZ=\lip_m\ZZ/m.\]
In a nutshell:

\begin{prp}\label{25.02.2025--7}We have an isomorphism of group schemes $\left(\Si_D^+ \right)^{\rm pf}\simeq\Sigma_D\star\wh{\bb Z}$. \qed
\end{prp}

Let us now {\it drop  the assumption that $D$ is connected}. 
Using again that the functor $(-)^{\mm{pf}}$ commutes with co-products, we derive from Corollary \ref{22.04.2025--4} the following consequence.

\begin{thm}\label{22.04.2025--1}We decompose $D$ into connected components, $D=D_1\sqcup\cdots\sqcup D_m$. Then the pro-finite group scheme     $\pi(\ct_{D,\bb{pt}})^{\mathrm{pf}}$ is isomorphic to \[ \wh\ZZ^{\star(m-1)} \star \Si_{D_1}\star \cdots \star \Si_{D_m}, 
\]
where we adopt the convention that $\wh\ZZ^{\star0}=\{e\}$.
\end{thm}

\begin{rmk} \label{25.02.2025--5}We briefly remove all assumptions on the field $k$. 
Let $G$ be a  group scheme   which is local, that is,   $G$ only has one closed point. Let $G=\lip_iG_i$, where the canonical morphisms $G\to G_i$ are all  faithfully flat and where $G_i$ is algebraic over $k$, cf.  \cite[3.3 and 14.1]{waterhouse79}.  Then $G_i$ is also local  and being of finite type, must be finite over $k$  \ega{I}{}{6.4.4}.
\end{rmk}

\subsection{Conclusion: Determination of $\pi(X)$ when $C$ is reduced}\label{24.04.2025--2}

Let us now remove the assumption that $C=\mathbf{pt}$ and that $D$ is connected; we suppose only that $C=C_1\sqcup\cdots\sqcup C_\ell$, where  $C_i=\mathbf{pt}$. Then, for each $i\in\{1,\ldots,\ell\}$, we let $D_i$ be the inverse image of $C_i$ in $D$, and write $D_i=D_{i,1}\sqcup \cdots\sqcup D_{i,m_i}$, where each $D_{i,j}$
is a finite and local scheme. On the other hand, we add the assumption that {\it all essentially finite vector bundles on $Y$ are in fact trivial}, i.e.  $\pi(Y)=0$. 

It then follows that under the equivalence $\bb{VB}(X)\stackrel\sim\to\bb{VB}(C)\ti_{\bb{VB}(D)}\bb{VB}(Y)$, each $\ce\in\bb{EF}(X)$ has its image in    
$\bb{VB}(C)\ti_{\bb{VB}(D)}\bb{VB}^{\mm{tr}}(Y)\stackrel\sim\to\ct_{D,C}$. 
We then arrive at a {\it faithfully flat} morphism of   group schemes 
\begin{equation}\label{18.02.2025--1}\pi(\ct_{D,C} )\aro\pi(X), 
\end{equation}
cf. \cite[Lemma 2.1]{biswas-hai-dos_santos21}. 
As $\pi(X)$ is pro-finite, this morphism shall factor as $ \pi(\ct_{D,C})\to\pi (\ct_{D,C})^{\mm{pf}}\to\pi(X)$.
It is not difficult to see that the resulting morphism \[\pi(\ct_{D,C})^{\mm{pf}}\aro\pi(X)\] is in fact an {\it isomorphism}. Indeed, a finite representation of $\pi(\ct_{D,C})$ must come from an object of  $\bb{EF}(X)$. 
We can now express our findings in the following synthetic form.

\begin{cor}\label{03.03.2025--4}Suppose that $\pi(Y)=\{0\}$. Then 
\[\begin{split}
\pi(X)&\simeq  \pi(\ct_{D_1,C_1})^{\mathrm{pf}}\star \cdots \star \pi(\ct_{D_\ell,C_\ell})^{\mathrm{pf}} 
\\
&\simeq \bigstar_{i=1}^\ell  \wh\ZZ^{\star(m_i-1)} \star\Si_{D_{i,1}}\star \cdots \star \Si_{D_{i,m_i}} .
\end{split}
\]\qed
\end{cor}
From Corollary \ref{03.03.2025--4}, we see that the  heart of the computation of $\pi(X)$ is  the determination of $\Sigma_D$ for a finite and  local $D$. This shall occupy the rest of the paper.

\begin{rmk}If $\mm{Pic}(X)$ contains a copy of $k^*$, then the morphism in eq. \eqref{18.02.2025--1} fails to be an isomorphism:   pick an invertible sheaf  on $X$   which becomes trivial on $Y$,  and  such that its class in $k^*\subset\mm{Pic}(X)$ is not of finite oder. 
\end{rmk}

\subsection{A cautionary example}
It is   natural to inquire if $\pi(X,x_0)$
can be described as efficiently as it can in the etale case even by  {\it dropping  the assumption  $\pi(Y)=0$} made in  Corollary  \ref{03.03.2025--4}. More precisely, we have in mind the analogous situation of   Theorem 4.12 and   Corollary 5.4 in  \cite[Exp. IX]{SGA1} which we briefly recall. To simplify, let us suppose that $D$ is either $\mathbf{pt}\sqcup\bb {pt}$ or   $\spc k[t]/(t^2)$,  and that $C=\mathbf{pt}$. Then, according to \cite{SGA1}, we have an isomorphism $
\pi^{\mathrm{et}}(Y)\star \wh\ZZ\simeq \pi^{\mathrm{et}}(X)$
in the first case and $\pi^{\mathrm{et}}(Y) \simeq\pi^{\mathrm{et}}(X)$ in the second.  
(Here, the amalgamated product is taken in the category of pro-finite groups.) The following example gives some limits to these speculations.

\begin{ex}\label{05.03.2025--1}Let us suppose that    $C=\mathbf{pt}$ and that $\te:D\to Y$  is a tangent vector at a point $y\in Y(k)$, i.e. $D=\spc k[t]/(t^2)$. 
(Since we assume that $D$ is a closed subscheme, this tangent vector cannot vanish.)
In addition, let us assume the existence of a   finite group scheme $G$ and  a principal $G$-bundle $Z\to Y$ such that, for any $z\in Z(k)$ of image $y$, the derivative $T_zZ\to T_yY$ vanishes. This is the case, for example, when $Y$ is an abelian variety, $G$ is the kernel of multiplication by $p$, and $Z=Y\to Y$ is multiplication by $p$.

 We contend that in this situation {\it it is not} possible to obtain an isomorphism \[\pi(Y)\star S\simeq \pi(X)\]
 where $S$ is some pro-fintie group scheme. More precisely, let $i:\pi(Y)\to \pi(Y)\star S$ be the canonical morphism  and suppose that there exists a pro-finite group scheme $S$ and an  isomorphism $\ph:\pi(Y)\star S\stackrel\sim\to\pi(X)$ such that $\nu_\#=\ph\circ i$. Let now $\ps:\pi(X)\to \pi(Y)$ be the morphism defined by means of the universal property of $\pi(Y)\star S$ applied to $\id:\pi(Y)\to \pi(Y)$ and the trivial morphism $S\to \pi(Y)$. Then, $\ps\nu_\#=\id_{\pi(Y)}$ and we arrive at a functor $\Ps:\bb{EF}(Y)\to\bb{EF}(X)$
 such that $\nu^*\circ \Ps:\bb{EF}(Y)\to\bb{EF}(Y)$ is naturally isomorphic, as a tensor functor,  to  $\id_{\bb{EF}(Y)}$.
Hence, the tensor functors 
\[
F:\bb{EF}(Y)\aro \modules{\co(D)},\quad\ce\longmapsto \Ga(D,\ce|_D)
\]
and 
\[
G:\bb{EF}(Y)\aro \modules{\co(D)},\quad\ce\longmapsto \ce|_{y}\ot\co(D)
\]
are isomorphic because $F\circ\nu^*$ and $G\circ\nu^*$ are likewise.

  Then   $Z|_{D}\to D$ is isomorphic to the principal $G$-bundle $Z|_{y}\times D\to D$ so that  $\te:D\to Y$ lifts to $\wt \te:D\to Z$, hence obtaining a contradiction.

\end{ex}

\section{Non-commutative Witt groups: The works of Ditters and Newman}\label{05.02.2025--1}

As we mentioned in the Introduction, one of the difficulties behind the effective calculation of fundamental group schemes lies in the fact that we lack enough identifiable local  group schemes. In this section, we make a brief presentations on a theory which allows us to obtain very relevant local group schemes and which seems to have drawn, unfortunately, little attention: this is the work of Ditters and Newmann \cite{ditters69,newman74}. This section makes no use of constructions and notations of Sections \ref{19.02.2025--1} to \ref{17.04.2025--1} and the geometry of pinched schemes.

\subsection{The algebra  $\cz$ of Ditters}\label{26.02.2025--1}
We fix $m$ a positive integer. 
Let $\cz$, respectively $\mathcal Z(m)$,  be the associative algebra (over $k$) on the variables $\{Z_i\}_{i=1}^\infty$, respectively $\{Z_i\}_{i=1}^m$.  Clearly, we regard $\mathcal Z(m)$ as a subalgebra of $\mathcal Z$. 
By convenience let us write   $Z_0=1$. On   $\mathcal Z$, we then introduce the structure of   bialgebra by decreeing that comultiplication $\De$ satisfies 
\[\Delta Z_h=  \sum_{ i+j=h  }Z_i\ot Z_j,\]
and that the  co-unit maps $Z_1,Z_2,\ldots$ to $0$. 
The bialgebra $\mathcal Z$ is sometimes referred to as the {\it Leibniz algebra} or the {\it Hopf algebra of non-commutative symmetric functions}, or the {\it universal non-commutative group-coalgebra (UNG)}.

In  \cite[1.2.7]{ditters69}, it is proved that  
once we define elements $\{S_i\}_{i=1}^\infty$ by the equations $S_0=1$ and  $\sum_{i=0}^nS_iZ_{n-i} =0$ for $n\ge1$, 
then
the   morphism of algebras $S:\cz\to \cz$ sending    $Z_i$ to $S_i$ endows $\cz$ with the structure of a {\it cocommutative Hopf-algebra}. In addition,  $\cz(m)$ is also a sub-Hopf-algebra. 

A final piece of structure on $\cz$ is its grading, which is obtained by decreeing   each $Z_i$ to be homogeneous of degree $i$.

\begin{notation}For each   integer $n>0$, let $\|n\|$ stand for  $[\log_pn]$, that is, $p^{\|n\|}\le n<p^{\|n\|+1}$.  
\end{notation}

\begin{dfn}[{\cite[1.4.8]{ditters69}}]Let $H\in\mathbf{Hpf}^{\mathrm{coc}}$. Given $\ell\in\NN^* \cup\{\infty\}$,       a curve of length $\ell$ on  $H$ is a sequence  of elements $(c_1,\ldots,c_\ell)\in H^{\ti \ell}$ such that, after fixing $c_0=1$, we have      $\De c_j= \sum_{i=0}^{j}c_i\ot c_{j-i}$, for every $1\le j\le \ell$. A curve of length $\ell$ is also called a divided power sequence of length $\ell$, cf. \cite[1.23]{newman74}. The set of curves of length $\ell$ on  $H$ shall be denoted by $\mm{Curv}_\ell(H)$.  
\end{dfn}

\begin{dfn}A curve $(c_i)_{i=1}^\ell$ is said to be minimal if $k\{c_1,\ldots,c_i\}= k\{c_1,\ldots,c_{p^{\|i\|}}\}$
for  $i=1,\ldots,\ell$. 
\end{dfn}

Note that, if $(c_i)_{i=1}^\ell$ is a minimal curve, then in fact $k\{c_1,\ldots,c_i\}$ can be generated simply by the elements indexed by power of $p$, i.e. $k\{c_1,\ldots,c_i\}=k\{c_1,c_p,\ldots,c_{p^{\|i\|}}\}$.

\begin{rmk}In \cite[Definition 6.1]{ditters75}, Ditters puts forward the notion of a {\it pure} curve in $\cz$. Our definition of ``minimal'' seeks to imitate this in a different setting, viz. for general Hopf algebras and without mention of weights. 
\end{rmk}

To every arrow $\ph:H\to H'$ in $\mathbf{Hpf}^{\mathrm{coc}}$ we have an associated map $\mm{Curv}_\ell(H)\to \mm{Curv}_\ell(H')$ and in this way $H\mapsto\mm{Curv}_\ell(H)$ becomes a functor $\mathbf{Hpf}^{\mathrm{coc}}\to\bb{Set}$. It is clear that $\mm{Curv}_\ell$ is then represented by $\cz(\ell)$.

\begin{thm}[Minimal curves {\cite[Theorem 2.1.4]{ditters69}}]
\label{minimal_curve} There exists  a minimal curve $(E_i)_{i=1}^\infty$ in $\cz$ with the additional following properties.  
\begin{enumerate}[(a)]
\item  Each 
$E_i$ is homogeneous of degree $i$. 
\item For each $i\ge0$, we have   $\deg(E_{p^i}-Z_{p^i})< p^i$. 
\end{enumerate}
\end{thm}

Let us now fix a curve $(E_i)_{i\ge1}$ as in Theorem \ref{minimal_curve}.  It is clear that $k\{E_1,\ldots,E_{p^\ell}\}\subset \cz(p^\ell)$ is a sub-Hopf-algebra of $\cz(p^\ell)$. 

\begin{dfn}\label{nonCommuattiveWitt}For each $\ell\ge0$, we define
\[\mathcal{NW}_{\ell+1}=k\{E_{p^0},\ldots,E_{p^\ell}\}=k\{E_{p^0},E_p,\ldots,E_{p^\ell}\}.
\]
This shall be called the {\it  non-commutative Witt Hopf-algebra} of length $\ell+1$. (See Section \ref{26.02.2025--2} for an explanation concerning the name.) The associated   group scheme $\spc \cn\cw_{\ell+1}^\circ$ shall be denoted by $\bm{N\!W}_{\ell+1}$, and will be called the {\it non-commutative infinitesimal Witt group scheme}. 
\end{dfn}

Note that it is a priori not clear that $\mathcal{NW}_{\ell+1}$ does not depend on the minimal curve chosen. This shall be cleared below. 

An obvious property of   $\bm {N\!W}_{\ell+1}$ is  the following. (For the concept of height, see \cite[II.7.1.4--6]{demazure-gabriel70}.)

\begin{lem}The group scheme $\bm{N\!W}_{\ell+1}$ is local and of height $\le\ell+1$. 
\end{lem} \begin{proof}Let $\mm{Ver}$ stand for the Verschiebung morphism of the cocommutative coalgebra $\cn\cw_{\ell+1}$ \cite[Ch. 2, Section 5.3, 112ff]{abe80}. Since $(E_i)_{i=1}^\infty$ is a curve in $\cz$, then 
\[\mathrm{Ver}^{i+1}(E_{p^i})=0,\]
see \cite[Lemma 2.5.8]{abe80}. 
Hence, $\mm{Ver}^{\ell+1}(\cn\cw_{\ell+1})=k$. From the canonical formula 
\[ f(\mm{Ver}(a)) =  f^p( a)^{1/p},\qquad \text{for each $f\in \cn\cw_{\ell+1}^*$ and $a\in\cn\cw_{\ell+1}$},
\]
we conclude that the augmentation ideal $(k1)^\perp\subset\cn\cw_{\ell+1}^\circ$ is annihilated by $\mm{Fr}^{\ell+1}$, which implies that $\cn\cw^\circ_{\ell+1}$ is local of height at most $\ell+1$.  
\end{proof}

The curve $(E_i)_{i=1}^\infty$ is not the unique one enjoying the properties of Theorem \ref{minimal_curve}; on the other hand, given any minimal curve $(E_i')_{i=1}^\infty$ satisfying (a) and (b) of Theorem \ref{minimal_curve}, the following can be said. 
If  \[\ph:\cz\aro k\{E_i'\,:\,i\ge1\}\]
is defined by  $\ph(Z_i)=E_i'$, then \[k\{E_i\,:\,i\ge1\}\arou\ph k\{E_i'\,:\,i\ge1\}\]
is an isomorphism. See \cite[Lemma 3.1.2, p. 51]{ditters69}. 
  
\begin{rmk}\label{25.04.2025--1}In \cite[3.1.3, p.52]{ditters69}, the Hopf algebra $k\{E_i,i\ge1\}$ is called the ``non-commutative exponential''. In \cite{newman74}, what we called $\cc{NW}_{\ell+1}$ is called    $\g P_\ell$. We have adopted the shifted index in order to     accord  with the notation for Witt vectors. 
\end{rmk}

\subsection{Relation to  the Witt group scheme  }\label{26.02.2025--2}

Let $\ell$ be a positive  integer. 
In order to highlight the importance of the Hopf algebra  $\mathcal{NW}_\ell$ and the group scheme $\spc\,\mathcal{NW}_{\ell}^\circ$, we relate it to the ``usual  Witt Hopf algebra.'' This was already remarked by both Ditters and Newman in a different setting. In order to provide reliable arguments, we require  a certain number of ``well-known'' results on  the duality theory of Hopf algebras, which are  explained in Section \ref{public_utility1}.

Let $\cw_\ell:=\mathcal{NW}^{\mm{ab}}_{\ell}$ be the cocommutative and commutative Hopf algebra constructed in Section \ref{public_utility1}. We shall explain its relation to the   Hopf algebra of  the Witt group scheme $\bm W_{\ell}$  \cite[V.1.6, p.543]{demazure-gabriel70}. More precisely,   let $S_j\in \FF_p[x_0,\ldots,x_j\,;\,y_0,\ldots,y_j]$ stand for the additive Witt polynomials \cite[V.1.4, p.541-2]{demazure-gabriel70} and let $\co(\bm W_{\ell}):=k[x_0,\ldots,x_{\ell-1}]$ be given the structure of Hopf algebra coming from the Witt   group scheme $\bm W_{\ell}$ of dimension $\ell$ over $k$, i.e.  comultiplication is 
\[X_{j}\longmapsto S_j(x_0 \ot1,\ldots,x_j\ot1\,;\,1\ot x_0,\ldots,1\ot x_j).
\]Then, either from the work of Dieudonn\'e  \cite[3.1.5-8]{ditters69} or from the appendix of \cite{newman74}, it follows that \[\co(\bm W_{\ell})\simeq\mathcal{W}_{\ell}.\] (We note that in \cite[Appendix]{newman74}, Newman does not follow the standard convention concerning Witt vectors.)
It is important  to  analyse not only  $\mathcal  W_{\ell}$ but also the Hopf dual $\mathcal  W_{\ell}^\circ$, since the pro-finite group scheme  $\bm{N\!W}_{\ell}$ will have $\spc \mathcal  W_{\ell}^\circ$ as its largest {\it abelian} quotient (see Lemma \ref{public_utility2}). The structure of the group scheme 
$\spc\co(\bm W_{\ell})^\circ$ should be known, but we find it hard to give a clear reference. 

Let us begin by employing 
\cite[Theorem 9.5]{waterhouse79} to obtain an isomorphism 
\[
\spc\co(\bm W_{\ell})^\circ\simeq \mm{Diag}(\La)\ti U,
\]
 where  $\La$ is an abelian group and $U$ is an  unipotent group scheme. Clearly $\La$ is just the group of characters of $\spc\co(\bm W_{\ell})^\circ$, which is the group of group-like elements of   $\co(\bm W_{\ell})^\circ$. From \cite[p. 129]{abe80}, $\La$  is isomorphic to  the group of $k$-points of $\bm W_{\ell}$, i.e.  $\La\simeq  \bm W_{\ell}(k)$.

Employing   Lemma \ref{public_utility2}, we conclude that $\co(U)$ is isomorphic to the Hopf subalgebra $\co(\bm W_{\ell})^\dagger$ of $\co(\bm W_{\ell})^\circ$. 
Now, we make use of the fact that we are in positive characteristic; letting $\g A_{\ell}$ stand for the augmentation ideal of $\co(\bm W_{\ell})$, it follows that \[\co(\bm W_{\ell})^\dagger=\lid_h\left(\co(\bm W_{\ell})/\g A_{\ell}^{[p^h]}\right)^*.
\] 
The determination of Hopf algebra $\left(\co(\bm W_{\ell})/\g A_{\ell}^{[p^h]}\right)^*$ is well-known as it enters the context of Cartier duality of finite group schemes. Indeed, let 
$\bm I_h$ stand for functor associating to a group scheme its $h$th Frobenius kernel; we have 
$\co(
\bm I_h\bm W_{\ell})=\co(\bm W_{\ell})/\g A_{\ell}^{[p^h]}$.
Then   
 \cite[Theorem, p.61]{demazure72} says that there exists a canonical isomorphism 
\[\bm I_h\bm W_{\ell}\simeq \bm I_{\ell}\bm W_h.
\] We then conclude that $U\simeq\lip_h\bm I_{\ell}\bm W_h$, i.e. \[U\simeq\bm I_{\ell}\bm W. \]
In conclusion: 
\begin{prp}\label{23.04.2025--3}The   abelian group scheme $\bm {N\!W}_{\ell}^{\mm{ab}}$ is isomorphic to  $\mm{Diag}(\bm W_{\ell}(k))\times\bm I_{\ell}\bm W$.\qed
\end{prp}

\section{The group   scheme $\Sigma_D$ and    the Hopf algebra $\bm H_D$ when $D$ is local}
\label{14.02.2025--1}
We now regain the setting and notations of Sections \ref{19.02.2025--1} to \ref{17.04.2025--1}, and in addition  suppose that   \[A:=\co(D)\] is a local ring with  maximal ideal $\g m\not=0$.  We find convenient to write $\cs_A$ instead of $\cs_D$ in what follows. In order to gain knowledge about the group scheme  $\Sigma_D$ of Definition \ref{24.02.2025--1}, we shall find an equivalence  \[\cs_A\arou\sim\modules{\bm H_A},\]
for a certain {\it associative algebra  $\bm H_A$},   
introduce on  $\bm H_A$ a   bialgebra structure and then relate it  to the {\it tensor product on $\cs_A$}. Moving on,  we relate   $\bm H_A$ 
and the non-commutative  Witt Hopf-algebras of  Definition \ref{nonCommuattiveWitt} by means of an  encompassing structure result of \cite{newman74}. (Newman's result  is also in the direction of the  ``Campbell-Hausdorff-Dieudonn\'e''  Theorem in \cite[Theorem 3.5.4]{ditters69}.)

\subsection{The bialgebra  $\bm H_A$ and the category $\cs_A$}
\label{21.02.2025--7}Using the decomposition $A=k1\op \g m$, we shall identify  $\g m^*$ with a subspace of  $A^*$ in what follows. 
  
Using the obvious bijections   
\[\begin{split}\mm{Hom}_A(V\ot A,V\ot A)&\stackrel\sim\to\mm{Hom}_k(V,V\ot A)
\\
&\stackrel\sim\to\mm{Hom}_k(A^*,\mm{End}_k(V)),
\end{split}\] 
and the restriction map  $\mm{Hom}_k(A^*,\mm{End}_k(V))\to \mm{Hom}_k(\g m^*,\mm{End}_k(V))$, 
we can associate to each  $(V,\ph)\in\mm{Obj}\,\cs_A$ an element of $\mm{Hom}_k(\g m^*,\mm{End}_k(V))=\mm{Hom}_{\Alg k}(\bm T(\g m^*),\mm{End}_k(V))$, that is, a structure of $\bm T(\g m^*)$-module on $V$. 
Letting $\eta\in A^*$ stand for the canonical map $A\to A/\g m=k$, we observe that  $\ph\in\mm{Hom}_k(A^*,\mm{End}_k(V))$ gives rise to an object of $\cs_A$ if and only if  $\ph(\eta)=\id_V$.  From this, the identity functor of $\vect k$ gives rise to   an equivalence of  
$k$-linear categories 
\begin{equation}\label{17.03.2025--1}
\Om:\cs_A\stackrel \sim\aro\modules {\bm T(\g m^*)}.
\end{equation}
In    down-to-earth terms: Let $\{a_i\}_{i=1}^n$ be a basis of $\g m$ with dual basis  $\{e_i\}_{i=1}^n$. Given  $(V,\ph)\in\mm{Obj}\,\cs_A$,  we  write $\ph\in\mm{Hom}_A(V\ot A,V\ot A)$ as 
$\ph(v\ot1)=v\ot1+\sum_{i=1}^n\ph_i(v)\ot a_i$. Then $\Om(V,\ph)$ is the $\bm T(\g m^*)$-module defined by   the morphism of $k$-algebras 
\[
\bm T(\g m^*)\aro\mm{End}_k(V),\quad e_i\longmapsto\ph_i.
\] 
We now define 
\[\boxed{\bm H_A=\bm T(\g m^*).}\] We are now required to endow $\bm H_A$ with the structure of a cocommutative bialgebra. Let    $\ep:\bm H_A\to k$ be the map of rings defined by $\ep(1)=1$ and   $\ep(\g m^*)=\{0\}$.   Let $\De_0:\g m^*\to\g m^*\ot\g m^*$ be  the transpose of multiplication $\g m\ot\g m\to\g m$ and let   $\De:\bm T(\g m^*)\to\bm T(\g m^*)\ot\bm T(\g m^*)$ be the map of   $k$-algebras associated to the linear map   
\[
\De_0+\id\ot1+1\ot\id:\g m^*\aro\bm T(\g m^*)\ot\bm T(\g m^*).
\]
Explicitly, this says the following.  If $c_{ij}^h$ are the structure constants of multiplication $\g m\ot\g m\to\g m$, i.e.  $a_i a_j=\sum_{h=1}^n
 c_{ij}^h a_h$, then  $\ep(e_i)=0$ and 
\[
\De(e_h)=e_h\ot1+1\ot e_h+\sum_{i,j}c^h_{ij}\,e_i\ot e_j.
\]

We write $A^*=k\po\eta\op\g m^*$ and let  \begin{equation}\label{15.04.2025--1}f:A^*\aro\bm H_A\end{equation} be the linear map  which sends $\eta$ to $1$ and $\g m^*$ identically  to the canonical copy of $\g m^*$ in $\bm T(\g m^*)$. 
If $\de$ is the comultiplication  and $\mm{ev}_1:A^*\to k$ is the co-unit  of $A^*$, then $\mm{ev}_1=\ep f$ and $f\ot f \circ \de=\De\circ f$. These formulas, together with the fact that $\g m^*$ generates $\bm H_A$,  allow us to show that $\De$  and $\ep$ endow $\bm H_A$ with the structure of a coalgebra, and   $f:A^*\to\bm H_A$ is a morphism of such. Since $\ep$ and $\De$ are maps of algebras, {\it $\bm H_A$ is now a cocommutative  bi-algebra}. 

These very natural  bialgebras are unfortunately not  well 
documented in the literature on the theme, except in the theory of the ``non-commutative symmetric functions'', as we see in the following.  
 
\begin{ex}Assume that $A=k[t]/(t^{m+1})$, where $m\ge1$. Letting $a_i=t^i$, so that $\{a_i\}_{i=1}^m$   is a basis of $\g m$, 
it follows that \[c_{ij}^h=\left\{\begin{array}{ll}0,&h\not=i+j,  \\
1,&  h=i+j.
\end{array}\right.\]   Then
$\De e_h=e_h\ot1+1\ot e_h+\sum_{i=1}^{h-1}e_i\ot e_{h-i}$,
and $\bm H_A$ is just  the truncated algebra of non-commutative symmetric functions    $\cz(m)$, cf. Section \ref{26.02.2025--1}. 
\end{ex}

The above constructions are   related to the tensor structure on $\cs_A$. Indeed,  let $(V,\ph)$ and $(W,\ps)$ be objects of $\cs_A$, where $\ph\in\mm{Hom}_k(V,V\ot A)$ and $\ps\in\mm{Hom}_k(W,W\ot A)$. The  tensor product in $\cs_A$,    $(V,\ph)\otimes(W,\ps)=(V\ot W,\ph\boxtimes\ps)$,   is obtained   by fixing  \[\ph\boxtimes\ps\in\mm{Hom}_k(V\ot W,V\ot W\ot A)\] as the composition 
\begin{equation}\label{05.10.2024--1}\xymatrix{
V\ot W\ar[r]^-{\ph\ot\ps}&V\ot A\ot W\ot A\ar@{=}[r]& V\ot W\ot A\ot A\ar[rr]^-{\id\ot\id\ot\mm{mult}}&&V\ot W\ot A}.
\end{equation}In explicit terms, if   $\ph=\id_V\ot1+\sum_i\ph_i\ot a_i$ and $\ps=\id_W\ot1+\sum_i\ps_i\ot a_i$,  then 
\[
\ph\boxtimes\ps=\id_{V\ot W}\ot1+\sum_h\left(\id_V\ot\ps_h+\ph_h\ot\id_W+\sum_{i,j}c_{ij}^h \,\ph_i\ot\ps_j \right)\ot a_h.
\]
Hence,  the  $\bm H_A$-module   $\Om[(V,\ph)\ot(W,\ps)]$ is given by 
the composition 
\[\xymatrix{\bm H_A\ar[r]^-{\De}& \bm H_A\ot \bm H_A\ar[r]&\mm{End}(V)\ot\mm{End}(W)\ar[rr]^-{\text{natural}}&&\mm{End}(V\ot W)}.
\]
It is also clear that, letting $\bb I$ stand for the unit object of $\cs_A$, then  $\Om(\bb I)$  is the  $\bm H_A$-module defined by $\ep:\bm H_A\to k= \mm{End}_k(k)$. This means that if we endow $\modules{\bm H_A}$ with its tensor structure obtained from $\De$ and $\ep$, cf.  \cite[III.5]{kassel95}, then $\Om$ is a tensor functor. 

Consequently, because of the equivalence of tensor categories 
$\modules{\bm H_A}\simeq \comodules{\bm H_A^\circ}$, we conclude that \[\co(\Si_D)\simeq \bm H_A^\circ.\]

\subsection{Pointed irreducible and cocommutative bialgebras}
The bialgebra $\bm H_A$ has already appeared in   \cite{newman74}.
For the sake of the reader with little experience in the theory of coalgebras, let us make a brief introduction to some basic material from \cite{newman74}.

Recall from \cite[p. 80]{abe80} that a coalgebra $L$ is said to be irreducible if its coradical $\mathrm{corad}\,L$ is a simple coalgebra. As  
the Jacobson radical of $A$ is $\g m$, it follows  that  $(\mathrm{corad}\,A^*)^\perp=\g m$ \cite[2.3.9(i), p.84]{abe80}. Consequently, 
 $\mathrm{corad}\,A^*=k\eta$  and   {\it $A^*$ is irreducible}. It is in addition {\it pointed} \cite[p. 80]{abe80}.

Following  \cite{newman74}, we abbreviate ``pointed, irreducible and cocommutative'' to ``PIC'',  so that $A^*$ is a PIC coalgebra. Using the map $f$ of eq. \eqref{15.04.2025--1}, we regard  $A^*$  as a subcoalgebra of  $\bm H_A$ (and the element $\eta\in A^*$ is the unit of  $\bm H_A$). Let $\bm H_A(1)$ be the irreducible component of $\bm H_A$ containing $1$ \cite[p.97]{abe80}: from \cite[Theorem 2.4.5(ii), p.96]{abe80},   $\bm H_A(1)$ is the sum of all irreducible subcoalgebras of $\bm H_A$ containing $1$. In particular, $A^*\subset \bm H_A(1)$ and $\bm H_A(1)$ contains algebra generators of $\bm H_A$. Since 
  $\bm H_A(1)$ is a {\it subalgebra} also \cite[Theorem 2.4.27p.105]{abe80}, it follows that $\bm H_A=\bm H_A(1)$. In summary, $\bm H_A$ is also PIC. 
  A fundamental result in the theory now  says that $\bm H_A$  already {\it comes with an antipode}, see p. 71 and Proposition 9.2.5 of \cite{sweedler69}, and hence is a {\it Hopf algebra}. (Note that no particular property about the field $k$ is employed here.)

 As noted in \cite[Remark, p.4]{newman74},    $\bm H_A$ is in fact the solution to a certain   universal problem in the following sense.   
Let the category of PIC coalgebras, respectively Hopf algebras, be  denoted by $\mathbf{picCog}$, respectively $\mathbf{picHpf}$.

\begin{prp}\label{25.04.2025--2}The forgetful functor $\mathbf{picHpf}\to \mathbf{picCog}$ has a left adjoint $\cf:\mathbf{picCog}\to \mathbf{picHpf}$ and $\mathcal F(A^*)=\bm H_A$. \qed
\end{prp}
(We note that Newman uses ``$F$'' instead of ``$\cf$''; we made the change to avoid confusion with the Frobenius morphisms.)
The above characterisation of $\bm H_A$ allowed Newman to prove a very useful structure result.

\subsection{The structure of $\bm H_A$: Newman's Theorem  }\label{07.05.2025--1}
We shall require the theory of the Verschiebung morphism of cocommutative coalgebras 
\cite[Chapter 2, Section 5.3, 112ff]{abe80}. 
Let 
\[h=\min\{i\in\NN\,:\,\mathrm{Fr}^i(\g m)=0\}.\] (By assumption $h\ge1$.)
From   the duality equation for 
$\mm{Ver} :A^*\to A^*$,   
\[ \la(\mm{Ver} (a)) =  \la^p( a)^{1/p},\qquad \text{for each $\la\in (A^*)^*$ and $a\in A^*$},
\]   we have   \[\left(  \mathrm{Fr}^i(A^{**})\right)^\perp=\mm{Ker}\,\mm{Ver}^i.\] (This also holds for the cocommutative algebra $\bm H_A$.)
This gives rise to   a filtration 
\begin{equation}\label{14.03.2025--1}
0=K_0 \subset \underbrace{\mm{Ker}\,\mm{Ver}}_{K_1}\subset \cdots\subset\underbrace{\mm{Ker}\,\mm{Ver}^h}_{K_h}=(A^*)^+,
\end{equation}
where we have followed custom  and written $(-)^+$ for the augmentation ideal of a coalgebra.
Because $\bm H_A$ is generated, as a $k$-algebra by   $f[(A^*)^+]$, the equality $\mm{Ver}\circ f=f\circ \mm{Ver}$ holds true,  and    $\mm{Ver}:{\bm H_A}\to {\bm H_A}$ is a morphism of rings,  we 
derive claims (i)--(iii) of the following: 
\begin{lem} \label{24.02.2025--2}(i) The Verschiebung morphism   $\mm{Ver}^h:{\bm H_A}\to{\bm H_A}$ annihilates the augmentation ideal $(\bm H_A)^+$,  (ii) while $\mm{Ver}^{h-1}$ does not. In particular, (iii) $\mm{Fr}^h$ annihilates   $(\bm H_A^\circ)^+=\{\ph\,:\,\ph(1)=0\}$. (iv) The morphism $\mm{Fr}^{h-1}$ does not annihilate   $(\bm H_A^\circ)^+$.
\end{lem}
\begin{proof}Only (iv) remains to be checked, and this follows from the fact that the map of algebras $\bm H_A^\circ\to (A^*)^*=A$ is surjective since $\bm H_A=\bm T(\g m^*)$. 
\end{proof}
  Another consequence of the filtration \eqref{14.03.2025--1} is that     $A^*$ comes equipped with a {\it regular basis}  in the sense of \cite[Definition 2.1]{newman74}. We recall what this means.  A basis $\cb$ of $A^*$ is regular when (i) $\mm{Ver}(\cb)\subset\cb\cup\{0\}$, (ii) if $\mm{Ver}(\la)=\mm{Ver}(\la')$, then   either $\la=\la'$ or $\mm{Ver}(\la)=0$ and (iii)   $\cb\setminus(A^*)^+\subset\{\eta\}$. Now, to construct a regular basis, we only need to employ the filtration $\{K_i\}$ and proceed as in the construction of a Jordan basis for a nilpotent linear map. 
    
From Theorem \cite[Theorem 2.16]{newman74} and Proposition \ref{25.04.2025--2}, we can assure that in the category  $\bb{Hpf}$, the Hopf algebra $\bm H_A$ is a coproduct (cf. section \ref{public_utility2}) of non-commutative Witt Hopf algebras, that is,  
\begin{equation}\label{25.02.2025--9}
\bm H_A\simeq   \mathcal{NW}_{\ell_1}\sqcup\cdots\sqcup\mathcal{NW}_{\ell_m}.
\end{equation}
  In addition, 
$\max_i\{\ell_i\}\le h$  and  for $\ell\ge1$, the number of copies of $\cn\cw_{\ell}$ in the above decomposition can be determined in the following way. (The reader is asked to bear in mind the difference between our convention and that of \cite{newman74}, see Remark \ref{25.04.2025--1}.) A   regular sequence in $\cb$ is a sequence  $(b_0,\ldots,b_r)$ of elements in   $\cb$ such that  $\mm{Ver}(b_i)=b_{i-1}$ for $i=1,\ldots,r$, and $\mm{Ver}(b_0)=0$. Then, the number of copies of $\cn\cw_\ell$ in the decomposition \eqref{25.02.2025--9} above is  the number of maximal regular sequences of    $\ell$ {\it elements}.  

\begin{cor}If $h=1$,  then $\bm H_A$ is a coproduct of copies of $\mathcal{NW}_1\simeq k[E_1]$ and hence is the universal enveloping algebra of a free Lie algebra on $\dim \g m$ generators. 
\end{cor}
\begin{proof}In this case, $\mm{Ver}(e_i)=0$ for all $1\le i\le n$ and $(e_1),\ldots,(e_n)$ are maximal regular sequences. 
\end{proof}

Let us now translate these findings in terms of the group scheme $\Sigma_D$ of Definition \ref{24.02.2025--1}. 
Since the functor    
$(-)^\circ: \bb{Hpf} \to\bb{Hpf}^{\mathrm{opp}}$  
is left-adjoint to    $(-)^\circ: \bb{Hpf}^{\mathrm{opp}} \to\bb{Hpf}$, cf.  the paragraph following \cite[2.3.14]{abe80}, 
we conclude from  general nonsense \cite[V.5]{maclane98} that $(-)^\circ$
sends co-products in $\bb{Hpf}$
to products in $\bb{Hpf}$. 
Hence, since $\co(\Sigma_D)=\boldsymbol H_A^\circ$, we conclude from eq. \eqref{25.02.2025--9} that 
$\co(\Sigma_D)\simeq \mathcal{NW}_{\ell_1}^\circ\sqcap\cdots\sqcap\mathcal{NW}_{\ell_m}^\circ$, 
where we use the notation $\sqcap$ for the product in $\bb{Hpf}$. More significantly, we then have.
\begin{cor}\label{23.04.2025--4} The group scheme  $\Sigma_D\simeq\spc\bm H_A^\circ$ is isomorphic to the amalgamated product \[ 
\bm{N\!W}_{\ell_1}\star\cdots\star\bm{N\!W}_{\ell_m}.
\]
The number of copies of $\bm{N\!W}_\ell$ equals $2\dim K_\ell-\dim K_{\ell+1}-\dim K_{\ell-1}$, where $K_i=\mm{Ker}\,\mm{Ver}^i$. 
\end{cor}

Let us end with a summary result. For that, we  abandon the assumption made on $D$, but refrain from making explicit the components of $\bm{N\!W}_\ell$ appearing in the decomposition. 

\begin{cor}\label{23.04.2025--5}Let $D$, $Y$, $X$ and  $C$ be as in Section \ref{19.02.2025--1}, and assume that $C$ is reduced while  $\pi(Y)=0$. Then $\pi(X)$ is isomorphic to the amalgamated product of copies of $\wh \ZZ$ and of $\bm {N\!W}_\ell$. 
\end{cor}



\begin{thebibliography}{99}

\bibitem[A80]{abe80}
E. Abe,  {\it  Hopf algebras.}
Cambridge Tracts in Mathematics, 74.
Cambridge University Press, Cambridge-New York, 1980.



\bibitem[AB16]{antei-biswas16}M. Antei and I. Biswas, {\it On the fundamental group scheme of rationally chain-connected varieties}.  
Int. Math. Res. Not. 2016, No. 1, 311-324 (2016).

\bibitem[BMM96]{BMM96}
Beidar, K. I.; Martindale, W. S., III; Mikhalev, A. V.
\textit{Rings with generalized identities}
Monogr. Textbooks Pure Appl. Math., 196, 1996, xiv+522 pp.


\bibitem[BH07]{biswas-holla07}I.  Biswas and Y.  Holla, {\it Comparison of fundamental group schemes of a projective variety and an ample hypersurface}. 
J. Algebraic Geom. 16 (2007), no. 3, 547--597.




\bibitem[BKP24]{biswas-kumar-parameswaran24}I. Biswas,  M. Kumar and A. J.  Parameswaran, {\it
Bertini type results and their applications.}  
Acta Math. Vietnam. 49, No. 4, 649-665 (2024).



\bibitem[BHdS21]{biswas-hai-dos_santos21}I. Biswas, P. H. Hai and J. P. dos Santos, {\it On the fundamental group schemes of certain quotient varieties.}
Tohoku Math. J. (2) 73(4), pp. 565-595 (2021) 

\bibitem[Bi09]{biswas09}I. Biswas, {\it On the fundamental group-scheme}. 
Bull. Sci. Math. 133, No. 5, 477-483 (2009).

\bibitem[Bri15]{brion15}M. Brion,
\textit{Which algebraic groups are Picard varieties?}, Sci. China, Math. 58, No. 3, 461--478 (2015). 

\bibitem[Da24]{das24}
 S. Das, 
{\it Galois covers of singular curves in positive characteristics}. 
Isr. J. Math. 263, No. 1, 397-432 (2024).



\bibitem[DM82]{deligne-milne82} P. Deligne and J. Milne, {\it Tannakian categories}.  Lecture Notes in Mathematics 900, pp. 101 -- 228, Springer-Verlag, Berlin-New York, 1982.




\bibitem[De72]{demazure72}M. Demazure,  
{\it Lectures on $p$-divisible groups}.
Lecture Notes in Math., Vol. 302
Springer-Verlag, Berlin-New York, 1972. 

\bibitem[DG70]{demazure-gabriel70}M. Demazure and P. Gabriel, \emph{Groupes alg\'ebriques. Tome I}. Masson \& Cie, \'Editeur, Paris; North-Holland Publishing Co., Amsterdam, 1970.  





\bibitem[DW23]{deninger-wibmer23}Ch. Deninger and M. Wibmer, {\it On the proalgebraic fundamental group of topological spaces and amalgamated products of affine group schemes}. Preprint arXiv:2306.03296 [math.AG]. 


\bibitem[Di73]{dieudonne73}
J. Dieudonn\'e, \textit{Introduction to the theory of formal groups}
Pure Appl. Math., 20, Marcel Dekker, Inc., New York, 1973.

\bibitem[Dit75]{ditters75}B. Ditters, {\it Groupes formels}. Publ. Math. Orsay  no. 149-75.42  (1975).  


\bibitem[Dit69]{ditters69}
E. Ditters, {\it Curves and exponential series in the theory of noncommutative formal groups}.  Thesis. Catholic University of Nijmegen, 1969. Available from \url{https://repository.ubn.ru.nl/}. 

\bibitem[Fe02]{ferrand02}D. Ferrand, {\it Conducteur, descente et pincement}. Bull. Soc. math. France 131(4), 2003, 553--585.

\bibitem[GW04]{goodearl-warfield04}K. R. Goodearl and R. B. Warfield, {\it An introduction to noncommutative Noetherian rings}. 
London Math. Soc. Stud. Texts, 61
Cambridge University Press, Cambridge, 2004.



\bibitem[HdS18]{hai-dos_santos18}P. H. Hai and J. P. dos Santos, {\it The action of the etale fundamental group scheme on the connected component of the essentially finite one}.  Math. Nachr. Vol. 291, Issue11-12, August 2018, 1733--1742. 

\bibitem[Ha93]{hain93} R. Hain, {\it Completions of  mapping class groups and the cycle $C-C^-$}. Contemporary Mathematics, Vol. 150, 75--105 (1993).



\bibitem[Haz78]{hazewinkel78}M. Hazewinkel, {\it Formal groups and applications}.
Pure and Applied Mathematics. 78. New York-San Francisco-London: Academic Press.  (1978).

\bibitem[Ho12]{howe} S. Howe, {\it Higher genus counterexamples to relative Manin-Mumford}. Master's thesis from ALGANT. 2012. 

\bibitem[Ja03]{jantzen03} J. C. Jantzen, {\it Representations of algebraic groups}. Second edition. Mathematical Surveys and Monographs, 107. American Mathematical Society, Providence, RI, 2003.



\bibitem[Mac98]{maclane98}S. Mac Lane, {\it Categories for the working mathematician}, Graduate Texts in Mathematics 5, Second edition, Springer-Verlag, New York, 1998. 

\bibitem[MM65]{milnor-moore65}J. Milnor and J. C.  Moore, {\it 
On the structure of Hopf algebras}. Ann. Math. (2) 81, 211-264 (1965).



\bibitem[Me11]{mehta11}V. B. Mehta, {\it Some Further Remarks on the Local Fundamental Group Scheme}. \verb+arXiv:1111.1074v1 [math.AG]+

\bibitem[Mo93]{montgomery93} S. Montgomery. \textit{Hopf Algebras and Their Actions on Rings}.
CBMS Regional Conference Series in Mathematics
Volume: 82; 1993; 238 pp.





\bibitem[Ne74]{newman74}K. Newman,
{\it The structure of free irreducible, cocommutative Hopf algebras}.
J. Algebra 29, 1-26 (1974).






\bibitem[No82]{nori82} M. V. Nori, \emph{The fundamental group scheme},  Proc. Indian Acad. Sci. Math. Sci. 91 (1982), no. 2, 73--122.

\bibitem[No76]{nori76} M. V. Nori, \emph{On the representations of the fundamental group}, Compositio Math. 33 (1976), no. 1, 29--41.








\bibitem[Sch95]{schneider95}H.-J. Schneider, {\it Lectures on Hopf algebras}. Notes by Sonia Natale.  Trab. Mat., 31/95. U. Nac. C\'ordoba, Fac. Mat., Astr. y F\'is., C\'ordoba, Argentina,  1995.  


\bibitem[Se88]{serre88}J.-P. Serre, {\it Algebraic groups and class fields}. Lecture Notes in Mathematics 177, Springer. 1988. 


\bibitem[SP]{stacks_project}The stacks project authors, {\it Stacks Project}.  

\bibitem[Sw69]{sweedler69}M. E. Sweedler.  \textit{Hopf algebras}. 
Mathematics Lecture Note Series, W. A. Benjamin, Inc., New York, (1969).










  





\bibitem[SGA1]{SGA1} {\it Rev\^etements \'etales et groupe fondamental}. 
S\'eminaire de g\'eom\'etrie alg\'ebrique du Bois Marie 1960--61. Directed by A. Grothendieck. With two papers by M. Raynaud. Updated and annotated reprint of the 1971 original. Documents Math\'ematiques 3. Soc. Math.  France, Paris, 2003.





\bibitem[OTZ22]{otabe-tonini-zhang22}
S. Otabe, F. Tonini and L.  Zhang, 
{\it
A generalized Abhyankar's conjecture for simple Lie algebras in characteristic $p>5$.
}
Math. Ann. 383, No. 3-4, 1721-1774 (2022).


\bibitem[Ka95]{kassel95}Ch. Kassel, {\it Quantum Groups}. Graduate Texts in Mathematics 155. Springer, 1995. 

\bibitem[Ot18]{otabe18}S. Otabe, {\it
On a purely inseparable analogue of the Abhyankar conjecture for affine curves}.
Compos. Math. 154 (2018), no. 8, 1633-1658.

\bibitem[Wa79]{waterhouse79} William C. Waterhouse, {\it Introduction to affine group schemes}. Graduate Texts in Mathematics, 66. Springer-Verlag, New York-Berlin, 1979. 


\end{thebibliography}
\end{document}